\documentclass[final,hidelinks]{siamart220329}

\usepackage{amsfonts,amssymb,enumitem}
\usepackage{graphicx,float}

\newsiamremark{remark}{Remark}

\headers{Error estimates of EWI for NLSE}{W. Bao, and C. Wang}

\title{Optimal error bounds on the exponential wave integrator for the nonlinear Schr\"odinger equation with low regularity potential and nonlinearity\thanks{Submitted to the editors DATE.
		\funding{This research is supported by the Ministry of Education, Singapore, under its Academic Research Fund MOE-T2EP20122-0002 (A-8000962-00-00).}}}

\author{Weizhu Bao\thanks{Department of Mathematics, National University of Singapore, Singapore 119076
		(\email{matbaowz@nus.edu.sg}, \url{https://blog.nus.edu.sg/matbwz/}).}
	\and Chushan Wang\thanks{Department of Mathematics, National University of Singapore, Singapore 119076
		(\email{e0546091@u.nus.edu}).}}

\usepackage{amsopn}

\newcommand{\R}{\mathbb{R}}
\newcommand{\Z}{\mathbb{Z}}
\newcommand{\Cl}{{C_\text{Lip}}}
\newcommand{\CL}{{C_L}}
\newcommand{\C}{\mathbb{C}}
\newcommand{\vx}{\mathbf{x}}
\newcommand{\rmd}{\mathrm{d}}

\newcommand{\vphi}{\varphi}
\newcommand{\F}{G}

\newcommand{\Phih}{\Phi_h}
\newcommand{\Phiht}{\Phi_h^{\langle \tau \rangle}}
\newcommand{\psin}[1]{\psi^{[#1]}}
\newcommand{\en}[1]{e^{[#1]}}
\newcommand{\psihn}[1]{\psi^{#1}}
\newcommand{\ehn}[1]{e^{#1}}
\newcommand{\psihhn}[1]{\psi^{\langle #1 \rangle}}
\newcommand{\ehhn}[1]{e^{\langle #1 \rangle}}

\begin{document}

\maketitle

\begin{abstract}
	We establish optimal error bounds for the exponential wave integrator (EWI) applied to the nonlinear Schr\"odinger equation (NLSE) with $ L^\infty $-potential and/or locally Lipschitz nonlinearity under the assumption of $ H^2 $-solution of the NLSE. For the semi-discretization in time by the first-order Gautschi-type EWI, we prove an optimal $ L^2 $-error bound at $ O(\tau) $ with $ \tau>0 $ being the time step size, together with a uniform $ H^2 $-bound of the numerical solution. For the full-discretization scheme obtained by using the Fourier spectral method in space, we prove an optimal $ L^2 $-error bound at $ O(\tau + h^2) $ without any coupling condition between $\tau$ and $h$, where $ h>0 $ is the mesh size. In addition, for $ W^{1, 4} $-potential and a little stronger regularity of the nonlinearity, under the assumption of $ H^3 $-solution, we obtain an optimal $ H^1 $-error bound. Furthermore, when the potential is of low regularity but the nonlinearity is sufficiently smooth, we propose an extended Fourier pseudospectral method which has the same error bound as the Fourier spectral method while its computational cost is similar to the standard Fourier pseudospectral method. Our new error bounds greatly improve the existing results for the NLSE with low regularity potential and/or nonlinearity. Extensive numerical results are reported to confirm our error estimates and to demonstrate that they are sharp.
\end{abstract}



\section{Introduction}
In this paper, we consider the following nonlinear Schr\"odinger equation (NLSE)
\begin{equation}\label{NLSE}
	\left\{
	\begin{aligned}
		&i \partial_t \psi(\vx, t) = -\Delta \psi(\vx, t) + V(\vx) \psi(\vx, t) + f(|\psi(\vx, t)|^2) \psi(\vx, t), \quad \vx \in \Omega, \  t>0, \\
		&\psi(\vx, 0) = \psi_0(\vx), \quad \vx \in \overline{\Omega},
	\end{aligned}
	\right.
\end{equation}
where $ t $ is time, $\vx = (x_1, \cdots, x_d)^T \in \R^d \  (d=1, 2, 3) $ is the spatial coordinate, $ \psi=\psi(\vx, t) $ is a complex-valued wave function, and $ \Omega = \Pi_{i=1}^d (a_i, b_i) \subset \R^d $ is a bounded domain equipped with periodic boundary condition. Here, $ V = V(\vx):\Omega \rightarrow \R $ is a 
real-valued potential and $ f=f(\rho):[0, \infty) \rightarrow \R $ with $ \rho=|\psi|^2 $ being the density describes the nonlinear interaction. We assume that $ V \in L^\infty(\Omega) $ and $ f(|z|^2)z:\C \rightarrow \C $ is locally Lipschitz continuous, and thus both $ V $ and $ f $ may be of low regularity. 



When $ V(\vx) = |\vx|^2/2 $ and $ f(\rho) = \rho $, the NLSE \cref{NLSE} collapses to the nonlinear Schr\"odinger equation with harmonic potential and cubic nonlinearity (or smooth potential and nonlinearity) or the Gross-Pitaevskii equation (GPE), which has been widely adopted for modeling and simulation in quantum mechanics, nonlinear optics and Bose-Einstein condensation (BEC) \cite{review_2013,ESY,NLS}. For the smooth NLSE with sufficiently smooth initial data $ \psi_0 $, many accurate and efficient numerical methods have been proposed and analyzed in last two decades, including the finite difference method \cite{FD,bao2013,review_2013,Ant}, the exponential wave integrator \cite{bao2014,ExpInt,SymEWI}, the time-splitting method \cite{bao2003JCP,BBD,lubich2008,schratz2016,review_2013,Ant,BCF}, the finite element method \cite{FEM1,FEM2,FEM3,FEM4,henning2017}, etc. Recently, many works are done to analyze and design numerical methods for the cubic NLSE with low regularity initial data $ \psi_0 $ and {with/without potential (see \cite{schratz2016,splitting_low_reg,LRI,LRI_sinum,LRI_error,LRI_general,LRI_fulldisc,LRI_sec,tree1,tree2}} and references therein for other dispersive partial differential equations).

Arising from different physics applications, both $ V $ and $ f $ in \cref{NLSE} may be of low regularity. Typical examples of the low regularity $ L^\infty $-potential include, in many physical contexts, the square-well potential or step potential, which are discontinuous; in the study of BEC in different trapping shape, the power law potential $ V(\vx) = |\vx|^\gamma \ (\gamma > 0) $ \cite{poten_power_law,poten_power_law_2}, and in the analysis of Josephson effect and Anderson localization, some disorder potential \cite{poten_Josephson,poten_anderson}. Low regularity nonlinearity such as $ f(\rho) = \rho^\sigma \ (\sigma>0) $ or $ f(\rho) = \rho \ln \rho $ are considered in, e.g., the Schr\"{o}dinger-Poisson-X$ \alpha $ model \cite{bao2003,SPXalpha}, the Lee-Huang-Yang correction \cite{LHY} which is adopted to model and simulate quantum droplets \cite{QD1,QD2,QD3,QD4}, and the mean-field model for Bose-Fermi mixture \cite{Had,Cai}.


Most numerical methods for the cubic NLSE with smooth potential can be extended straightforwardly to solve the NLSE \cref{NLSE} with $ L^\infty $-potential and/or locally Lipschitz nonlinearity (different from the singular nonlinearity in \cite{sinum2019,bao2019,bao2022,bao2022singular}). However, the performance of these methods are quite different from the smooth case and the error analysis of them is a very subtle and challenging question. For \cref{NLSE} with power-type nonlinearity $ f(\rho) = \rho^\sigma $ and sufficiently smooth potential, the Lie-Trotter time-splitting method is analyzed in \cite{bao2023,choi2021,ignat2011} with reduced convergence order in $ L^2 $-norm when $ \sigma<1/2 $ and in $ H^1 $-norm when $ \sigma<1 $. The analysis of \cref{NLSE} with smooth nonlinearity and $ L^\infty $-potential seems more challenging and the only known convergence result is the one obtained in \cite{henning2017} for the Crank-Nicolson Galerkin scheme, where first order convergence in time and less than second order convergence in space in $ L^2 $-norm are shown under strong assumptions on the exact solution (among others $ \partial_t \psi \in H^2 $), and a coupling condition between the time step size $ \tau $ and the mesh size $ h $. Some low regularity integrators or resonance-based Fourier integrators are also proposed to reduce the regularity requirements on both $ V $ and $ \psi $, while the regularity assumption on $ V $ is still stronger than $ H^1 $ {\cite{zhao2021,tree2,alama2022}}, which still excludes the popular well potential and step potential widely adopted in physics literatures. The main difficulty comes from the low regularity of solution of the NLSE with $ L^\infty $-potential and locally Lipschitz nonlinearity, where only $ H^2 $ well-posedness is guaranteed \cite{kato1987,cazenave2003}, and the low regularity of the potential and the nonlinearity which causes order reduction in local truncation errors and prevent us from obtaining stability estimates in high order Sobolev spaces $ H^\alpha\ (\alpha>d/2) $ (see \cite{bao2023,zhao2021,henning2017} for more detailed discussion). Besides, for the NLSE \cref{NLSE} with purely $ L^\infty $-potential, how to estimate the spatial discretization is also a challenging problem and it turns out that it is very subtle and challenging to estimate the classical methods including the finite difference method, the pseudospectral method and the finite element method \cite{henning2017}.

The main aim of this paper is to establish optimal error bounds for a first order Gautschi-type exponential wave integrator (EWI), {also known as the exponential Euler scheme in the literature \cite{ExpInt}}, applied to the NLSE with $ L^\infty $-potential and/or locally Lipschitz nonlinearity. 
{Our main results are as follows: 
\begin{enumerate}[label=(\roman*),leftmargin=*]
	\item For the semi-discretization in time (EWI \cref{EWI}), we prove an optimal $ L^2 $-error bound at $ O(\tau) $ with $\tau>0$ being the time step size, and a uniform $ H^2 $-bound of the numerical solution, under the assumption of $ H^2 $-solution of the NLSE (see \cref{eq:semi_error_L2} in \cref{thm:error_estimates}). 
	\item For the full discretization of the EWI by using the Fourier spectral method for spatial derivatives (EWI-FS \cref{EWI-FS}), we prove an optimal $ L^2 $-error bound at $ O(\tau + h^2) $ without any coupling condition between $ \tau $ and the mesh size $ h $ (see \cref{eq:full_error_L2} in \cref{thm:error_estimates_FS}). 
	\item For $ W^{1, 4} $-potential and a little more regular nonlinearity, under the assumption of $H^3$-solution, we obtain optimal $ H^1 $-error bounds for EWI and EWI-FS schemes. (see \cref{eq:semi_error_H1} in \cref{thm:error_estimates} and \cref{eq:full_error_H1} in \cref{thm:error_estimates_FS}). 
	\item When the potential is of low regularity but the nonlinearity is sufficiently smooth, we propose an extended Fourier pseudospectral method for spatial discretization of the EWI, leading to EWI-EFP scheme \cref{EWI-EFP}. For the EWI-EFP, we establish optimal error bounds in $L^2$- and $H^1$-norm under the same assumption on the potential and exact solution as the EWI-FS (see \cref{coro:error_estimates_EFP}). However, the computational cost of EWI-EFP is similar to the standard Fourier pseudospectral discretization of the EWI. 
\end{enumerate}
}
Our error bounds greatly improve the previous results for the NLSE with low regularity potential and/or nonlinearity. {In general, compared with the error estimates of classical exponential wave integrators \cite{tree2} and time-splitting methods \cite{bao2023} in the literature}, to obtain optimal error bounds, we reduce the differentiability requirement on the potential by two orders and on the nonlinearity by one order. Moreover, when $ V \in L^\infty $ and $ f $ is smooth as considered in \cite{henning2017}, compared with their results for the Crank-Nicolson Galerkin scheme, we improve the convergence order in $ L^2 $-norm to the optimal first order in time and the optimal second order in space, remove the coupling condition requirement between $ \tau $ and $ h $ in \cite{henning2017}, relax the regularity assumption on the exact solution such that it is theoretically guaranteed, and reduce the computational cost in practical implementation.

Here, we briefly explain why we can obtain the improved error bounds. In general, time-splitting methods and EWIs require weaker regularity on the exact solution to obtain the same order of convergence, compared with finite difference methods. In practical computation, time-splitting methods tend to outperform EWIs when the solution
is smooth, which requires the potential and nonlinearity as well as the initial data are all smooth. The main reason is that time-splitting methods are usually structure-preserving scheme, i.e. they preserve mass conservation, time symmetry, time-transverse invariance, and dispersion relation at the discretized level 
\cite{Ant,review_2013}. On the contrary, when the NLSE \cref{NLSE} involves low regularity potential and/or nonlinearity, leading to a solution with low regularity,  we find that the first-order Gautschi-type EWI offers two major advantages in obtaining optimal error bounds: (i) in obtaining local truncation errors, time-splitting methods need to apply the Laplacian $ \Delta $ to the equation while the EWI only needs to apply $ \partial_t $ to the equation, and thus
the EWI needs weaker regularity requirement on both potential and nonlinearity; and (ii) a smoothing operator is adopted in the EWI scheme to control the dispersion of  high frequencies and thus it helps to keep the numerical solution in $ H^2 $ at each time step, which makes it possible to obtain the stability estimates in high order Sobolev spaces, while it is a challenging and subtle task to establish $H^2$-bounds of the numerical solution by using the time-splitting methods.

The rest of the paper is organized as follows. In Section 2, we present a semi-discretization in time by the first-order Gautschi-type EWI and then a full discretization in space by the Fourier spectral/extended pseudospectral methods.
Sections 3 and 4 are devoted to the error estimates of the semi-discretization scheme and the full-discretization scheme, respectively. Numerical results are reported in Section 5 to confirm the error estimates. Finally, some conclusions are drawn in Section 6. Throughout the paper, we adopt the standard Sobolev spaces as well as their corresponding norms, and denote by $ C $ a generic positive constant independent of the mesh size $ h $ and time step size $ \tau $, and by $ C(\alpha) $ a generic positive constant depending only on the parameter $ \alpha $. The notation $ A \lesssim B $ is used to represent that there exists a generic constant $ C>0 $, such that $ |A| \leq CB $.

\section{The exponential wave integrator Fourier spectral method}
In this secton, we introduce an exponential wave integrator and its spatial discretization to solve the NLSE with low regularity potential and nonlinearity.
For simplicity of the presentation and {to avoid heavy notations}, we only carry out the analysis in 1D and take $ \Omega = (a, b) $. {The only dimension sensitive estimates are the Sobolev embedding into $L^\infty$ and the inverse inequalities to control $L^\infty$-norm by $L^2$-norm. In our analysis, we only use the embedding $H^2(\Omega) \hookrightarrow L^\infty(\Omega)$ which holds for 1D, 2D and 3D, and for the inverse inequalities, we clearly show how it depends on the space dimension. Thus, generalizations to 2D and 3D are straightforward, and the main results remain unchanged.}

We define periodic Sobolev spaces as (see, e.g. \cite{feng2022} for the equivalent definition)
\begin{equation*}
	H_\text{per}^m(\Omega) := \{\phi \in H^m(\Omega) : \phi^{(k)}(a) = \phi^{(k)}(b), \ k=0, \cdots, m-1\}, \quad m \geq 1. 
\end{equation*}

\subsection{Semi-discretization in time by an exponential wave integrator}
Choose a time step size $ \tau > 0 $ and denote time steps as $ t_n = n\tau $ for $ n = 0, 1, \cdots $. By Duhamel's formula, the exact solution of the NLSE \cref{NLSE} is given as
\begin{eqnarray}
\psi(t_{n+1})=&&\psi(t_{n}+\tau)= e^{i \tau \Delta} \psi(t_n)\nonumber\\
		 &&- i \int_0^\tau e^{i(\tau - s)\Delta} \left [ V \psi(t_n + s) + f(|\psi(t_n + s)|^2)\psi(t_n + s) \right ] \rmd s,\quad n\ge0, \label{eq:Duhamel}
\end{eqnarray}
where we abbreviate $\psi(x,t)$ by $\psi(t)$ for simplicity of notations when there is no confusion.
Let $ \psin{n}:=\psin{n}(x) $ be the approximation of $ \psi(x, t_n) $ for $ n \geq 0 $. Applying the approximation $ \psi(t_n + s) \approx \psi(t_n) $ for the integrand in \eqref{eq:Duhamel} and integrating out $ e^{i(\tau - s)\Delta} $ exactly, we get a semi-discretization in time by the first-order Gautschi-type EWI as
\begin{equation}\label{EWI}
	\begin{aligned}
		\psin{n+1} &= \Phi^\tau(\psin{n}) := e^{i \tau \Delta} \psin{n} - i \tau \vphi_1(i \tau \Delta) \left( V \psin{n} + f(|\psin{n}|^2)\psin{n} \right), \quad n \geq 0, \\
		\psin{0} &= \psi_0,
	\end{aligned}
\end{equation}
where $ \vphi_1 $ is an entire function defined as
\begin{equation*}
	\vphi_1(z) = \frac{e^z - 1}{z}, \qquad z \in \C.
\end{equation*}
The operator $ \varphi_1(i \tau \Delta) $ is defined through its action in the Fourier space as
\begin{eqnarray}
	\left (\varphi_1(i \tau \Delta) v \right )(x)
	&&= \sum_{l \in \mathbb{Z}} \varphi_1(- i \tau \mu_l^2 ) \widehat{v}_l e^{i \mu_l (x-a)}  \nonumber\\
	&&= \widehat{v}_0 + \sum_{l \in \mathbb{Z} \setminus \{0\}} \frac{1-e^{-i\tau\mu_l^2}}{i \tau \mu_l^2} \widehat{v}_l e^{i \mu_l (x-a)},
	\quad x \in \Omega, \label{eq:phi1_def}
\end{eqnarray}
where $\mu_l = \frac{2 \pi l}{b-a}$ for $l\in{\mathbb Z}$, and $ \widehat{v}_l \ (l \in \mathbb{Z}) $ are the Fourier coefficients of the function $ v \in L^2(\Omega) $ defined as
\begin{equation}\label{eq:FT}
	\widehat{v}_l = \frac{1}{b-a} \int_a^b v(x) e^{-i \mu_l(x-a)} \rmd x, \quad l \in \mathbb{Z}.
\end{equation}
From \eqref{eq:phi1_def}, noting that $ | 1-e^{-i\theta} | \leq 2 $ for $ \theta \in \R $, we see that
\begin{equation}
	\left | \widehat{(\varphi_1(i \tau \Delta) v)}_l \right | \leq
	\left\{
	\begin{aligned}
		&\frac{2}{\tau} \frac{|\widehat{v}_l|}{\mu_l^2}, && l \in \mathbb{Z} \setminus \{0\}, \\
		&|\widehat{v}_0|, && l = 0,
	\end{aligned}
	\right.
\end{equation}
which implies $ \varphi_1(i \tau \Delta) v \in H^2_\text{per}(\Omega) $ for all $ v \in L^2(\Omega) $. Hence, $ \Phi^\tau $ is indeed a flow in $ H^2_\text{per}(\Omega) $ for any $ V \in L^\infty(\Omega) $, making it possible to obtain uniform $H^2$-bound of the semi-discrete solution with some new analysis techniques we will introduce later. 

In fact, the introduction of the smoothing function $ \varphi_1(i \tau \Delta) $ in
\eqref{EWI} is one of the major advantages of the EWI \cref{EWI} over the time-splitting methods in terms of controlling the dispersion of high frequencies or resonance. With this smoothing function, one can show that the numerical solution
is in $H^2$ at every time step. For comparison, based on the results in \cite{bao2023}
for time-splitting methods applied to the NLSE with semi-sooth nonlinearity, the numerical solution of the semi-discretization is not in $ H^2 $ in general! The situation is even worse if there is purely $ L^\infty $-potential.

\subsection{Full discretization by the Fourier spectral method}
Then we further discretize the semi-discretization \cref{EWI} in space by the Fourier spectral method to obtain a full-discretization scheme.
Choose a mesh size $ h=(b-a)/N $ with $ N $ being a positive integer and denote grid points as
\begin{equation*}
	x_j = a + jh, \quad j = 0, 1, \cdots, N.
\end{equation*}
Define the index sets
\begin{equation*}
	\mathcal{T}_N = \left\{-\frac{N}{2}, \cdots, \frac{N}{2}-1 \right\}, \quad \mathcal{T}_N^0 = \{0, 1, \cdots, N\},
\end{equation*}
and denote
\begin{align}
	&X_N = \text{span}\left\{e^{i \mu_l(x - a)}: l \in \mathcal{T}_N\right\},  \\
	&Y_N = \left\{ v=(v_0, v_1, \cdots, v_N)^T \in \C^{N+1}: v_0 = v_N \right\}.
\end{align}
Let $ P_N:L^2(\Omega) \rightarrow X_N $ be the standard $ L^2 $-projection onto $ X_N $ and $ I_N: Y_N  \rightarrow X_N $ be the standard Fourier interpolation operator as
\begin{align}
	&(P_N u)(x) = \sum_{l \in \mathcal{T}_N} \widehat u_l e^{i \mu_l(x - a)}, \\
	&(I_N v)(x) = \sum_{l \in \mathcal{T}_N} \widetilde v_l e^{i \mu_l(x - a)}, \quad x \in \overline{\Omega} = [a, b],
\end{align}
where $ u \in L^2(\Omega) $, $ v \in Y_N $, $ \widehat{u}_l \ (l \in \mathbb{Z}) $ are the Fourier coefficients of $ u $ defined in \cref{eq:FT} and $ \widetilde v_l \ (l \in \mathcal{T}_N) $ are the discrete Fourier transform coefficients defined as
\begin{equation}
	\widetilde{v}_l = \frac{1}{N} \sum_{j=0}^{N-1} v_j e^{-i \mu_l (x_j - a)}, \quad l \in \mathcal{T}_N.
\end{equation}
Let $ \psihn{n}: = \psihn{n}(x) $ be the approximation of $ \psi(x, t_n) $ for $ n \geq 0 $. Then an exponential wave integrator-Fourier spectral method (EWI-FS) for the NLSE  \eqref{NLSE} is given as
\begin{equation}\label{EWI-FS}
	\begin{aligned}
		\psihn{n+1} &= \Phih^\tau(\psihn{n}) := e^{i \tau \Delta} \psihn{n} - i \tau \vphi_1(i \tau \Delta) P_N \left( V \psihn{n} + f(|\psihn{n}|^2)\psihn{n} \right), \quad n \geq 0, \\
		\psihn{0} &= P_N \psi_0.
	\end{aligned}
\end{equation}
Note that $ \psihn{n} \in X_N $ for $ n \geq 0 $ and we have
\begin{equation}\label{psihn_coef}
	\begin{aligned}
		\widehat{(\psihn{n+1})}_l &= e^{-i \tau \mu_l^2} \widehat{(\psihn{n})}_l - i \tau \vphi_1(-i \tau \mu_l^2) \left( \widehat{(V \psihn{n})}_l + \widehat{\F(\psihn{n})}_l \right), \quad n \geq 0, \\
		\widehat{(\psihn{0})}_l &= \widehat{(\psi_0)}_l, \quad l \in \mathcal{T}_N,
	\end{aligned}
\end{equation}
where $ \F(\psihn{n})(x) = \F(\psihn{n}(x)) := f(|\psihn{n}(x)|^2)\psihn{n}(x) $ for $ x \in \Omega $. We remark here that the EWI-FS is usually implemented by the Fourier pseudospectral method (see, e.g., \cite{bao2014,feng2022}) in practical computations. Of course, due to the low regularity of the potential and/or nonlinearity, it is very hard to establish error bounds for the full-discretization by the Fourier pseudospectral method.

\subsection{Full discretization by an extended Fourier pseudospectral method}
In practice, the Fourier spectral method cannot be efficiently implemented. Here, we propose an extended Fourier pseudospectral method when the potential is of low regularity but the nonlinearity is sufficiently smooth, i.e., we adopt the Fourier spectral method to discretize the linear potential and use the Fourier pseudospectral method to discretize the nonlinearity. This full discretization has two advantages:
(i) we can establish its optimal error bounds, and (ii) the computational cost
of this discretization is similar to the standard Fourier pseudospectral method.

Let $ \psihhn{n}_j $ be the numerical approximation of $ \psi(x_j, t_n) $ for $ j \in \mathcal{T}_N^0 $ and $ n \geq 0 $, and denote $ \psihhn{n}:=(\psihhn{n}_0, \psihhn{n}_1, \cdots, \psihhn{n}_N)^T \in Y_N $. Then an exponential wave integrator-extended Fourier pseudospectral (EWI-EFP) method for the NLSE \eqref{NLSE} reads
\begin{equation}\label{EWI-EFP}
	\begin{aligned}
		\psihhn{n+1}_j
		&= \sum_{l \in \mathcal{T}_N} e^{-i \tau \mu_l^2} \widetilde{(\psihhn{n})_l} e^{i \mu_l(x_j-a)} \\
		&\quad - i \tau \sum_{l \in \mathcal{T}_N} \vphi_1(-i \tau \mu_l^2) \left( \widehat{\left (V I_N \psihhn{n} \right )_l} + \widetilde{\F(\psihhn{n})_l} \right) e^{i \mu_l(x_j-a)}, \quad n \geq 0, \\
		\psihhn{0}_j &= \psi_0(x_j), \quad j \in \mathcal{T}_N^0,
	\end{aligned}
\end{equation}
where $ \F(\psihhn{n})_j = f(|\psihhn{n}_j|^2)\psihhn{n}_j $ for $ j \in \mathcal{T}_N^0 $. To compute the Fourier projection coefficients $ \widehat{\left (V I_N \psihhn{n} \right )_l} $, we use an extended FFT as shown below. Note that $ I_N \psihhn{n} \in X_N $ for all $ n \geq 0 $, and thus we have
\begin{equation}\label{poten_trunc}
	P_N(V I_N \psihhn{n}) = P_N\left (P_{2N}(V) I_N \psihhn{n} \right ), \quad n \geq 0.
\end{equation}
Moreover, since $ P_{2N}(V) I_N \psihhn{n} \in X_{4N} $ and $ I_{4N} $ is an identity on $ X_{4N} $, we have
\begin{equation*}
	 P_{2N}(V) I_N \psihhn{n} = I_{4N}\left (P_{2N}(V) I_N \psihhn{n} \right ), \quad n \geq 0,
\end{equation*}
which plugged into \cref{poten_trunc} yields
\begin{equation}\label{eq:spectral_discretization_of_potential}
	P_N\left (V I_N \psihhn{n}\right ) = P_N I_{4N}\left (P_{2N}(V) I_N \psihhn{n} \right ), \quad n \geq 0,
\end{equation}
where $ P_{2N}(V) $ can be precomputed numerically or analytically, and thus the right hand side of \cref{eq:spectral_discretization_of_potential} can be computed exactly and efficiently using the extended FFT: using FFT for $ P_{2N}(V) I_N \psihhn{n} $ with length $ 4N $ instead of $ N $. As a result, the memory cost is $ O(4N) $ and the computational cost per time step is $ O(4N\log(4N)) $.
Note that $ \psihhn{n} \ (n \geq 0) $ obtained by \cref{EWI-EFP} satisfies
\begin{equation}\label{EWI-FSP_equation}
	\begin{aligned}
		I_N \psihhn{n+1} &= e^{i \tau \Delta} I_N \psihhn{n} - i \tau \vphi_1(i \tau \Delta) \left( P_N \left( V I_N \psihhn{n} \right) + I_N \F (\psihhn{n}) \right), \\
		I_N \psihhn{0} &= I_N \psi_0, \qquad n \geq 0.
	\end{aligned}
\end{equation}

\section{Optimal error bounds for the semi-discretization \eqref{EWI}}
In this section, we establish optimal error bounds in $L^2$-norm and $H^1$-norm
for the semi-discretization \eqref{EWI} of the NLSE \eqref{NLSE}.

\subsection{Main results}
For the optimal $ L^2 $-norm error bound, we assume that the nonlinearity is locally Lipschitz continuous, i.e., there exists a fixed function $ \Cl (\cdot): {\mathbb R}^+\to {\mathbb R}^+$ such that
\begin{equation}\label{A}
	| f(|z_1|^2)z_1 - f(|z_2|^2)z_2 | \leq \Cl(M_0) |z_1 - z_2|, \quad z_j \in \C, \  |z_j| \leq M_0, \quad j=1, 2. \tag{A}
\end{equation}
Assumption \cref{A} is satisfied by $ f \in C^1((0, \infty)) $ satisfying
	\begin{equation*}\label{A2}
		|f(\rho)| + |\rho f^\prime(\rho)| \leq L(M_0), \quad 0 < \rho \leq M_0
	\end{equation*}
	with $ \Cl(M_0) \sim L(M_0) $ for $ M_0 > 0 $. In particular, \cref{A} allows
	\begin{enumerate}
		\item[(i)] $ f(\rho) = \lambda_1 \rho^{\sigma_1} +\lambda_2 \rho^{\sigma_2} $ for any $ 0<\sigma_1<\sigma_2 $ and $\lambda_1,\lambda_2\in{\mathbb R}$ with $ \Cl(M_0) \sim |\lambda_1| M_0^{\sigma_1} +|\lambda_2| M_0^{\sigma_2} $; 
		\item[(ii)] $ f(\rho) = \lambda \rho^\sigma \ln \rho $ for any $ \sigma > 0 $
	and $\lambda\in {\mathbb R}$ with $ \Cl(M_0) \sim 1 + M_0^\sigma + M_0^\sigma |\ln M_0| $. 
	\end{enumerate}

For the optimal $ H^1 $-norm error bound, we assume
\begin{equation}\label{B}
	\|  f(|v|^2)v - f(|w|^2)w  \|_{H^1} \leq C(\| v \|_{H^3}, \| w \|_{H^2}) \| v-w \|_{H^1}, \  v \in H^3(\Omega), w \in H^2(\Omega). \tag{B}
\end{equation}
Assumption \cref{B} is satisfied by
	\begin{enumerate}
		\item[(i)] $ f(\rho) = \lambda_1\rho^{\sigma_1} +\lambda_2 \rho^{\sigma_2} $ for $ \sigma_2>\sigma_1 \geq 1/2 $ and $\lambda_1,\lambda_2\in{\mathbb R}$
		with $ C(\cdot,\cdot) $ depending on $ \| v \|_{H^3} $ and $ \| w \|_{H^2} $;
		\item[(ii)] $ f(\rho) = \lambda\rho^\sigma \ln \rho $ for any $ \sigma > 1/2 $
		and $\lambda\in{\mathbb R}$ with $ C(\cdot,\cdot) $ depending on $ \| v \|_{H^3} $ and $ \| w \|_{H^2} $. 
	\end{enumerate}
We remark here that \cref{B} implies \cref{A} by taking $ v(x) \equiv z_1 $ and $ w(x) \equiv z_2 $ in \cref{B}.
{Nonlinearity satisfying \cref{B} occurs in physical applications including Lee-Huang-Yang correction \cite{LHY,QD1,QD2,QD3,QD4} and Bose-Fermi mixture \cite{Had,Cai} in 1D, 2D and 3D, and Schrodinger-Poisson-X$\alpha$ model \cite{bao2003,SPXalpha} in 2D. Assumption \cref{A} covers, in addition to all those mentioned before, the case of Schrodinger-Poisson-X$\alpha$ model in 3D. }

Let $ T_\text{max} $ be the maximal existing time of the solution of the NLSE \cref{NLSE} and take $ 0 < T < T_\text{max} $ be a fixed time. Define
\begin{equation}\label{definM}
	M := \max\left\{\| \psi \|_{L^\infty([0, T]; H^2)}, \| \psi \|_{L^\infty([0, T]; L^\infty)}, \| \partial_t \psi \|_{L^\infty([0, T]; L^2)}, \| V \|_{L^\infty} \right\}.
\end{equation}

Let $ \psin{n} $ be the numerical approximation obtained by the EWI \cref{EWI}, then we have
\begin{theorem}\label{thm:error_estimates}
	Under the assumptions that $ V \in L^\infty(\Omega) $, $ f $ satisfies Assumption \cref{A} and the exact solution $ \psi \in C([0, T]; H_{\text{\rm per}}^2(\Omega)) \cap C^1([0, T]; L^2(\Omega)) $, there exists $ \tau_0>0 $ depending on $ M $ and $ T $ and sufficiently small such that for any $ 0 < \tau < \tau_0 $, we have $ \psin{n} \in H^2_\text{\rm per}(\Omega) $ for $ 0 \leq n \leq T/\tau $ and
	\begin{equation}\label{eq:semi_error_L2}
		\begin{aligned}
			&\| \psi(\cdot, t_n) - \psin{n}  \|_{L^2} \lesssim \tau, \quad \| \psin{n} \|_{H^2} \leq C(M), \\
			&\| \psi(\cdot, t_n) - \psin{n}  \|_{H^1} \lesssim \sqrt{\tau}, \qquad 0 \leq n \leq T/\tau.
		\end{aligned}
	\end{equation}
	Moreover, if $ V \in W^{1, 4}(\Omega) \cap H^1_\text{\rm per}(\Omega) $, $ f $ satisfies \cref{B} and $ \psi \in C([0, T]; H_\text{\rm per}^3(\Omega)) \cap C^1([0, T]; H^1(\Omega)) $, we have, for $ 0 < \tau < \tau_0 $,
	\begin{equation}\label{eq:semi_error_H1}
		\| \psi(\cdot, t_n) - \psin{n}  \|_{H^1} \lesssim \tau, \qquad 0 \leq n \leq T/\tau.
	\end{equation}
\end{theorem}
{
\begin{remark}
	According to the known regularity results (see, e.g. Corollary 4.8.6 in \cite{cazenave2003}), under the assumptions that $V \in L^\infty(\Omega)$ and \cref{A}, it can be expected that $ \psi \in C([0, T]; H_{\text{\rm per}}^2(\Omega)) \cap C^1([0, T]; L^2(\Omega)) $ for some $0<T<T_\text{max}$ if $\psi_0 \in H^2_\text{per}(\Omega)$. 
\end{remark}
}
\begin{remark}
	Recall that, for the time-splitting methods analyzed in \cite{bao2023} with $ f(\rho) = \rho^\sigma $, the optimal $ L^2 $-norm error bound in time is obtained for $ V \in H^2(\Omega)$ and $\sigma \geq 1/2 $, and the optimal $ H^1 $-norm error bound in time is obtained for $ V \in H^3(\Omega)$ and $\sigma \geq 1 $. Hence, our results greatly relax the regularity requirements on both the potential and nonlinearity.
\end{remark}

In the following, we shall prove \cref{thm:error_estimates}. We start with the proof of \cref{eq:semi_error_L2}, and the proof of \cref{eq:semi_error_H1} can be obtained by the standard Lady Windermere's fan argument with the established uniform $ H^2 $-bound of the semi-discretization solution in \cref{eq:semi_error_L2}.

In the rest of this section, we assume that $ V \in L^\infty(\Omega) $, $ f $ satisfies Assumption \cref{A} and $ \psi \in C([0, T]; H_{\text{\rm per}}^2(\Omega)) \cap C^1([0, T]; L^2(\Omega)) $.

\subsection{Local truncation error}
We define an operator $ B:L^\infty(\Omega) \rightarrow L^\infty(\Omega) $ as
\begin{equation}\label{eq:B_def}
	B(v) = V v + f(|v|^2)v, \quad v \in L^\infty(\Omega).
\end{equation}
and define a constant $ \CL(\cdot):= \| V \|_{L^\infty} + \Cl(\cdot) $ with $ \Cl(\cdot) $ given by Assumption \cref{A}. For the operator $ B $, we have
\begin{lemma}\label{lem:diff_B}
	Let $ v, w \in L^\infty(\Omega) $ satisfying $ \| v \|_{L^\infty} \leq M_0 $ and $ \| w \|_{L^\infty} \leq M_0 $, then
	\begin{equation}
		\| B(v) - B(w) \|_{L^2} \leq \CL(M_0) \| v-w \|_{L^2}.
	\end{equation}
\end{lemma}

\begin{proof}
	Recalling \cref{eq:B_def,A}, we have
	\begin{equation*}
		\begin{aligned}
			\| B(v) - B(w) \|_{L^2}
			&= \| V(v-w) + f(|v|^2)v - f(|w|^2)w \|_{L^2} \\
			&\leq \| V \|_{L^\infty} \| v-w \|_{L^2} + \Cl(M_0) \| v - w \|_{L^2} \\
			&= \CL(M_0) \| v-w \|_{L^2},
		\end{aligned}
	\end{equation*}
	which completes the proof.
\end{proof}

\begin{lemma}\label{lem:est_gn}
	For $ 0 \leq n \leq T/\tau-1 $, define
	\begin{equation}\label{eq:gn_def}
		\begin{aligned}
			g_n(t)
			&:= B(\psi(t_n+t)) - B(\psi(t_n)), \qquad 0 \leq t \leq \tau.
		\end{aligned}
	\end{equation}
	Then $ g_n \in C([0, \tau]; L^2(\Omega)) \cap W^{1, \infty}([0, \tau]; L^2(\Omega)) $ satisfies
	\begin{align}
		&\| g_n \|_{L^\infty([0, \tau];L^2)} \leq \CL(M) M \tau, \label{eq:est_g_n} \\
		&\| \partial_t g_n \|_{L^\infty([0, \tau];L^2)} \leq \CL(M)M. \label{eq:est_partial_t_g_n}
	\end{align}
\end{lemma}

\begin{proof}
	Using \cref{lem:diff_B}, we have, for $ 0 \leq s < t \leq \tau $,
	\begin{eqnarray}\label{eq:diff_gn}
		\| g_n(t) - g_n(s) \|_{L^2}
		&&= \|  B(\psi(t_n+t)) - B(\psi(t_n+s)) \|_{L^2} \nonumber\\
		&&\leq \CL(M) \| \psi(t_n+t) - \psi(t_n+s) \|_{L^2} \nonumber\\
		&&\leq \CL(M) \int_s^t \| \partial_t \psi(t_n+\sigma) \|_{L^2} \rmd \sigma.
	\end{eqnarray}
	From \eqref{eq:diff_gn}, recalling \cref{A}, one has $ g_n \in C([0, \tau]; L^2(\Omega)) $, and, by using \cref{lem:diff_B} again, one has
	\begin{equation*}
		\begin{aligned}
			\| g_n(t) \|_{L^2}
			&=\| B(\psi(t_n+t)) - B(\psi(t_n)) \|_{L^2} \leq \CL(M) \| \psi(t_n+t) - \psi(t_n) \|_{L^2} \\
			&\leq \tau \CL(M) \| \partial_t \psi \|_{L^\infty([t_n, t_n+\tau]; L^2)} \leq \CL(M) M \tau,
		\end{aligned}
	\end{equation*}	
	which proves \cref{eq:est_g_n}. Noting \eqref{eq:diff_gn}, from the standard theory of Sobolev spaces (see, e.g., Proposition 1.3.12 in \cite{cazenave2003}), we have $ g_n \in W^{1, \infty}([0, T]; L^2(\Omega)) $ and, by letting $ \varphi(\sigma) = \CL(M) \| \partial_t \psi(t_n+\sigma) \|_{L^2} $ for $ 0 \leq \sigma \leq \tau $,
	\begin{equation*}
		\| \partial_t g_n \|_{L^\infty([0, \tau]; L^2)} \leq \| \varphi \|_{L^\infty([0, \tau])} \leq \CL(M) M,
	\end{equation*}
	which concludes the proof.
\end{proof}

Similar to Lemma 4.8.5 in \cite{cazenave2003}, we have
\begin{lemma}\label{lem:H2_estimate}
	Assume $ \tau > 0 $ and $ g \in C([0, \tau]; L^2(\Omega)) \cap W^{1, 1}([0, \tau]; L^2(\Omega)) $. If
	\begin{equation}\label{eq:w_def}
		w(t) = -i \int_0^t e^{i(t-s)\Delta} g(s) \rmd s, \quad t \in [0, \tau],
	\end{equation}
	then we have
	\begin{equation}\label{eq:est_w}
		\| \Delta w \|_{L^\infty([0, \tau]; L^2)} \leq \| g \|_{L^\infty([0, \tau]; L^2)} + \| g(0) \|_{L^2} + \| \partial_t g \|_{L^1([0, \tau]; L^2)}.
	\end{equation}
\end{lemma}

\begin{proof}
	Taking the time derivative on both sides of \cref{eq:w_def} and noting that $ g \in W^{1, 1}([0, \tau]; L^2(\Omega)) $, we have for $ 0 \leq t \leq \tau $
	\begin{eqnarray}\label{eq:partialw}
		\partial_t w(t) = - i \frac{d}{dt} \int_0^t e^{is\Delta} g(t-s) \rmd s
		&&= - i  e^{it\Delta} g(0) - i \int_0^t e^{is\Delta} \partial_t g(t-s) \rmd s \nonumber\\
		&&= - i  e^{it\Delta} g(0) - i \int_0^t e^{i(t-s)\Delta} \partial_t g(s) \rmd s.
	\end{eqnarray}
	From \eqref{eq:partialw}, using the isometry property of $ e^{i t \Delta} $, we have
	\begin{equation}\label{eq:partialw_est}
		\| \partial_t w(t) \|_{L^2} \leq \| g(0) \|_{L^2} + \| \partial_t g \|_{L^1([0, \tau]; L^2)}, \quad 0 \leq t \leq \tau.
	\end{equation}
	Note that $ w $ defined in \cref{eq:w_def} satisfies the equation
	\begin{equation*}
		i \partial_t w = - \Delta w + g, \quad 0 \leq t \leq \tau,
	\end{equation*}
	which implies, by using \cref{eq:partialw_est},
	\begin{equation}
		\| \Delta w(t) \|_{L^2} \leq \| \partial_t w(t) \|_{L^2} + \| g(t) \|_{L^2} \leq \| g(0) \|_{L^2} + \| \partial_t g \|_{L^1([0, \tau]; L^2)} + \| g(t) \|_{L^2},
	\end{equation}
	and the conclusion follows from taking supremum of $ t $ on both sides.
\end{proof}

Now we can obtain the following local truncation error estimates.
\begin{proposition}[local truncation error]\label{prop:local_eror}
	For $ 0 \leq n \leq T/\tau -1 $, we have
	\begin{equation}\label{eq:local_error}
		\| \psi(t_{n+1}) - \Phi^\tau(\psi(t_n)) \|_{H^{\alpha}} \leq C(M) \tau^{2-\alpha/2}, \quad 0 \leq \alpha \leq 2,
	\end{equation}
	where $ C(M) \sim \CL (M)M $.
\end{proposition}

\begin{proof}
	Recalling \eqref{eq:Duhamel} and \eqref{eq:B_def}, we have
	\begin{equation}\label{eq:duhamel}
		\psi(t_{n+1}) = e^{i \tau \Delta} \psi(t_n) - i \int_0^\tau e^{i(\tau - s) \Delta}B(\psi(t_n+s))\rmd s, \quad 0 \leq n \leq T/\tau -1.
	\end{equation}
	By the construction of the EWI \cref{EWI} and \cref{eq:B_def}, we have
	\begin{equation}\label{eq:EWI}
		\Phi^\tau(\psi(t_n)) = e^{i \tau \Delta} \psi(t_n) - i \int_0^\tau e^{i(\tau - s)\Delta} B(\psi(t_n)) \rmd s, \quad 0 \leq n \leq T/\tau -1.
	\end{equation}
	Subtracting \cref{eq:EWI} from \cref{eq:duhamel} and recalling \cref{eq:gn_def}, we have
	\begin{eqnarray}\label{eq:local_error_integral}
		\psi(t_{n+1}) - \Phi^\tau(\psi(t_n))
		&&= - i \int_0^\tau e^{i(\tau - s) \Delta} (B(\psi(t_n+s)) - B(\psi(t_n)))\rmd s \nonumber\\
		&&= - i \int_0^\tau e^{i(\tau - s) \Delta} g_n(s) \rmd s, \quad 0 \leq n \leq T/\tau -1.
	\end{eqnarray}
	From \eqref{eq:local_error_integral}, using \cref{eq:est_g_n}, one gets
	\begin{equation}\label{eq:local_error_L2}
		\| \psi(t_{n+1}) - \Phi^\tau(\psi(t_n)) \|_{L^2} \leq \int_0^\tau \| g_n(s) \|_{L^2} \rmd s \leq \CL(M) M \tau^2,
	\end{equation}
	which proves \cref{eq:local_error} for $ \alpha=0 $. Then we shall establish \cref{eq:local_error} with $ \alpha=2 $, and \cref{eq:local_error} with $ 0 < \alpha< 2 $ will follow from the Gagliardo-Nirenberg interpolation inequalities. Applying \cref{lem:H2_estimate} to \eqref{eq:local_error_integral}, using \cref{eq:est_partial_t_g_n} and noting $ g_n(0) = 0 $, we have
	\begin{eqnarray}
			&&\| \Delta(\psi(t_{n+1}) - \Phi^\tau(\psi(t_n))) \|_{L^2} \nonumber\\
			&&\leq \| g_n \|_{L^\infty([0, \tau]; L^2)} + \| g_n(0) \|_{L^2} + \| \partial_t g_n \|_{L^1([0, \tau]; L^2)} \nonumber\\
			&&\leq \CL(M)M \tau + \tau \| \partial_t g_n \|_{L^\infty([0, \tau]; L^2)} \leq 2\CL(M)M \tau,
	\end{eqnarray}
	which combined with \cref{eq:local_error_L2} implies
	\begin{equation}\label{eq:local_error_H2}
		\| \psi(t_{n+1}) - \Phi^\tau(\psi(t_n)) \|_{H^2} \leq C(M) \tau, \quad 0 \leq n \leq T/\tau -1,
	\end{equation}
	where $ C(M) \sim \CL (M)M $. The conclusion follows from \cref{eq:local_error_L2,eq:local_error_H2} and the Gagliardo-Nirenberg interpolation inequalities.
\end{proof}

\begin{remark}
	In \cref{eq:local_error_L2}, the optimal local truncation error in $ L^2 $-norm is obtained with the boundedness of $ \| \partial_t B(\psi(t)) \|_{L^2} $ (recalling \cref{lem:est_gn}) instead of $ \| \Delta B(\psi(t)) \|_{L^2} $ in the time-splitting methods \cite{bao2023}.
\end{remark}

\subsection{$ L^\infty $-conditional stability estimate of \eqref{EWI}}
Then we shall establish the $ L^\infty $-conditional stability estimate of the numerical flow \eqref{EWI}. The key is the following lemma, which can be understood as the smoothing effect of the operator $ \varphi_1(i \tau \Delta) $, which is another major advantage of the EWI \cref{EWI}.
\begin{lemma}\label{lem:smoothingofphi1}
	Let $ v, w \in L^2(\Omega) $ and $ 0 < \tau < 1 $. Then we have
	\begin{equation*}
		\| \vphi_1(i\tau\Delta) v - \vphi_1(i\tau\Delta) w \|_{H^{\alpha}} \leq C(\alpha)\tau^{-\alpha/2} \| v - w \|_{L^2}, \quad 0 \leq \alpha \leq 2,
	\end{equation*}
	where $ C(\alpha) = 2^\frac{\alpha}{2}(1+\mu_1^{-2})^\frac{\alpha}{2} $.
\end{lemma}

\begin{proof}
	It suffices to show that for any $ v \in L^2(\Omega) $,
	\begin{equation}\label{eq:smoothingofphi1}
		\| \vphi_1(i\tau\Delta) v \|_{H^{\alpha}} \leq C(\alpha) \tau^{-\alpha/2} \| v \|_{L^2}, \quad 0 \leq \alpha \leq 2.
	\end{equation}
	Note that
	\begin{equation}\label{eq:expm1}
		| e^{i \theta} - 1 | \leq 2^{\gamma} \theta^{1-\gamma}, \quad \theta \in \R, \quad 0 \leq \gamma \leq 1.
	\end{equation}
	By Parseval's identity, using \cref{eq:expm1} with $ \gamma = \alpha/2 $ and recalling \eqref{eq:phi1_def}, we have
	\begin{equation*}
		\begin{aligned}
			\frac{1}{b-a}\| \vphi_1(i\tau\Delta) v \|_{H^{\alpha}}^2
			&= \sum_{l\in\Z} (1+\mu_l^2)^{\alpha} |\vphi_1(-i\tau \mu_l^2)|^2 | \widehat v_l |^2 \\
			&= | \widehat v_0 |^2 + \sum_{l\in\Z\setminus\{0\}} (1+\mu_l^2)^{\alpha} \left| \frac{e^{i\tau\mu_l^2} - 1}{\tau\mu_l^2} \right|^2 |\widehat v_l|^2 \\
			&\leq | \widehat v_0 |^2 + 2^{\alpha} \sum_{l\in\Z\setminus\{0\}} (1+\mu_l^2)^{\alpha} \left( \tau\mu_l^2 \right)^{-\alpha} |\widehat v_l|^2 \\
			&= | \widehat v_0 |^2 + 2^{\alpha} \tau^{-\alpha} \sum_{l\in\Z\setminus\{0\}} \left(\frac{1+\mu_l^2}{\mu_l^2}\right)^{\alpha}  |\widehat v_l|^2 \\
			&\leq | \widehat v_0 |^2 + C(\alpha)^2 \tau^{-\alpha} \sum_{l\in\Z\setminus\{0\}} |\widehat v_l|^2 \\
			&\leq C(\alpha)^2\tau^{-\alpha} \sum_{l\in\Z} |\widehat v_l|^2 = C(\alpha)^2 \tau^{-\alpha} \frac{1}{b-a} \| v \|_{L^2}^2,
		\end{aligned}
	\end{equation*}
	which proves \cref{eq:smoothingofphi1} and concludes the proof.
\end{proof}

With \cref{lem:smoothingofphi1}, we are able to obtain the stability estimate of the numerical flow \eqref{EWI} up to $ H^2 $ without additional regularity on the potential and nonlinearity.
\begin{proposition}[stability estimate]\label{prop:stability}
	Let $ v, w \in H_\text{\rm per}^2(\Omega) $ such that $ \| v \|_{L^\infty} \leq M_0 $ and $ \| w \|_{L^\infty} \leq M_0 $ and let $ 0 < \tau < 1 $. Then we have, for $ 0 \leq \alpha \leq 2 $,
	\begin{equation*}
		\| \Phi^\tau(v) - \Phi^\tau(w) \|_{H^\alpha} \leq \| v - w \|_{H^\alpha} + C(M_0) \tau^{1-\alpha/2} \| v - w \|_{L^2}.
	\end{equation*}
\end{proposition}

\begin{proof}
	Recalling \cref{EWI} and \cref{eq:B_def}, we have
	\begin{equation}\label{eq:Phi}
		\Phi^\tau(u) = e^{i \tau \Delta} u - i \tau \vphi_1(i \tau \Delta) B(u), \qquad u \in H^2_\text{per}(\Omega).
	\end{equation}
	Taking $ u = v $ and $ u = w $ in \cref{eq:Phi}, subtracting one from the other and using the isometry property of $ e^{it\Delta} $, \cref{lem:smoothingofphi1} and \cref{lem:diff_B}, we have
	\begin{align*}
		\| \Phi^\tau(v) - \Phi^\tau(w) \|_{H^\alpha}
		&\leq \| e^{i \tau \Delta} v - e^{i \tau \Delta} w \|_{H^\alpha} + \tau \| \vphi_1(i \tau \Delta) (B(v) - B(w)) \|_{H^\alpha} \\
		&\leq \| v - w \|_{H^\alpha} + C(\alpha) \tau^{1-\alpha/2} \| B(v) - B(w) \|_{L^2} \\
		&\leq \| v - w \|_{H^\alpha} + C(\alpha) \tau^{1-\alpha/2} \CL(M_0) \| v - w \|_{L^2}.
	\end{align*}
	The conclusion follows from letting $ C(M_0) = C(\alpha) \CL(M_0) $ with $ \alpha=2 $.
\end{proof}

\subsection{Proof of the optimal $ L^2 $-error bound \cref{eq:semi_error_L2}}
With the local truncation error estimate in Proposition \ref{prop:local_eror} and the $ L^\infty $-conditional stability estimate in Proposition \ref{prop:stability},
we can prove \cref{eq:semi_error_L2} by mathematical induction.

\begin{proof}[Proof of \cref{eq:semi_error_L2} in \cref{thm:error_estimates}]
	Define the error function $ \en{n} := \psi(t_n) - \psin{n} $ for $ 0 \leq n \leq T/\tau $. For $ 0 \leq n \leq T/\tau-1 $ and $ 0 \leq \alpha \leq 2 $, we have
	\begin{eqnarray}\label{eq:error_propogation}
		\| \en{n+1} \|_{H^\alpha}
		&&= \| \psi(t_{n+1}) - \psin{n+1}  \|_{H^\alpha} = \| \psi(t_{n+1}) - \Phi^\tau (\psin{n}) \|_{H^\alpha} \nonumber \\
		&&\leq \| \psi(t_{n+1}) - \Phi^\tau (\psi(t_n)) \|_{H^\alpha} + \| \Phi^\tau (\psi(t_n)) - \Phi^\tau (\psin{n}) \|_{H^\alpha}.
	\end{eqnarray}
	In the following, we first establish the error bounds in $ L^2 $-norm and $ H^\frac{7}{4} $-norm together by the mathematical induction, which, in particular, yield the uniform $ L^\infty $-bound of $ \psin{n} $. With the uniform $ L^\infty $-bound, we can obtain the uniform $ H^2 $-bound of $ \psin{n} $. The error bound in $ H^1 $-norm will follow from the error bound in $ L^2 $-norm and the uniform $ H^2 $-bound by using the standard interpolation inequalities.
	
	Let $ C_0 := \max\{C(1+M), M, 1\} \geq 1 $ with $M$ given in \eqref{definM} and $ C(\cdot) $ defined in \cref{prop:stability}, and $C_1:=C(M)$ with $C(\cdot)$ defined in \cref{prop:local_eror}. Let $ 0<\tau_0<1 $ be chosen such that
	\begin{equation}\label{eq:tau_0}
		2 C_0Te^{C_0T}C_1\tau_0^\frac{1}{8} \leq 1/c,
	\end{equation}
	where $ c $ is the constant given by the Sobolev embedding $ H^\frac{7}{4} \hookrightarrow L^\infty $. We are going to prove that when $ 0 < \tau < \tau_0 $, we have, for $ 0 \leq n \leq T/\tau $,
	\begin{equation}\label{eq:claim}
		\| \en{n} \|_{L^2} \leq e^{C_0T}C_1 \tau, \quad \| \en{n} \|_{H^\frac{7}{4}} \leq 2 C_0Te^{C_0T}C_1\tau^\frac{1}{8}.
	\end{equation}
	We shall prove \cref{eq:claim} by mathematical induction. When $ n = 0 $, $ \en{n} = \psin{0} - \psi_0 = 0 $, and $ \cref{eq:claim} $ holds trivially. We assume that \cref{eq:claim} holds for $ 0 \leq n \leq m \leq T/\tau -1 $. Under this assumption, we have, by Sobolev embedding, $ \tau < \tau_0 $ and \cref{eq:tau_0},
	\begin{equation}\label{eq:Linfty}
		\| \psin{n} \|_{L^\infty} \leq \| \psi(t_n) \|_{L^\infty} + \| \en{n} \|_{L^\infty} \leq M + c \| \en{n} \|_{H^\frac{7}{4}} \leq M + 1, \quad 0 \leq n \leq m.
	\end{equation}
	Taking $ \alpha = 0 $ and $ \alpha=7/4 $ in \eqref{eq:error_propogation}, we have for $ 0 \leq n \leq T/\tau-1 $,
	\begin{align}
		&\| \en{n+1} \|_{L^2} \leq \| \psi(t_{n+1}) - \Phi^\tau (\psi(t_n)) \|_{L^2} + \| \Phi^\tau (\psi(t_n)) - \Phi^\tau (\psin{n}) \|_{L^2}, \label{error_eq_1}\\
		&\| \en{n+1} \|_{H^\frac{7}{4}} \leq \| \psi(t_{n+1}) - \Phi^\tau (\psi(t_n)) \|_{H^\frac{7}{4}} + \| \Phi^\tau (\psi(t_n)) - \Phi^\tau (\psin{n}) \|_{H^\frac{7}{4}}. \label{error_eq_2}
	\end{align}
	Using \cref{prop:local_eror,prop:stability} with $ \alpha=0 $ and $ \alpha=7/4 $ for \cref{error_eq_1} and \cref{error_eq_2}, respectively, and noting \cref{eq:Linfty}, we have for $ 0 \leq n \leq m $,
	\begin{align}
		&\| \en{n+1} \|_{L^2} \leq (1+C_0 \tau) \| \en{n} \|_{L^2} + C_1 \tau^2, \label{eq:induction_1}\\
		&\| \en{n+1} \|_{H^\frac{7}{4}} \leq \| \en{n} \|_{H^\frac{7}{4}} + C_0 \tau^\frac{1}{8} \| \en{n} \|_{L^2} + C_1 \tau^{1+\frac{1}{8}}.  \label{eq:induction_2}
	\end{align}
	From \cref{eq:induction_1}, using the standard discrete Gronwall's inequality, we have
	\begin{equation}\label{eq:conclusion1}
		\| \en{m+1} \|_{L^2} \leq e^{C_0 T} C_1 \tau.
	\end{equation}
	From \cref{eq:induction_2}, using the assumption that \cref{eq:claim} holds for $ 0 \leq n \leq m $, we have
	\begin{equation}\label{eq:e_H7/4}
		\| \en{n+1} \|_{H^\frac{7}{4}} \leq \| \en{n} \|_{H^\frac{7}{4}} + C_0 \tau^\frac{1}{8}  e^{C_0 T}C_1\tau  + C_1 \tau^{1+\frac{1}{8}}, \quad 0 \leq n \leq m.
	\end{equation}
	Summing over $ n $ from $ 0 $ to $ m $ in \cref{eq:e_H7/4}, noting that $ \en{0}=0 $ and $ C_0 \geq 1 $, we obtain
	\begin{eqnarray}\label{eq:conclusion2}
			\| \en{m+1} \|_{H^\frac{7}{4}}
			&&\leq C_0 \tau^\frac{1}{8}  e^{C_0T}C_1 m\tau + C_1 m \tau^{1+\frac{1}{8}} \nonumber\\
			&&\leq C_0 T e^{C_0T} C_1 \tau^{\frac{1}{8}}+ C_1 T \tau^{\frac{1}{8}} \nonumber\\
			&&\leq 2 C_0 Te^{C_0T} C_1 \tau^\frac{1}{8}.
	\end{eqnarray}
	Combing \eqref{eq:conclusion1} and \eqref{eq:conclusion2}, we prove \cref{eq:claim} for $ n=m+1 $, and thus for all $ 0 \leq n \leq T/\tau $ by mathematical induction.
	
	Then we prove the uniform $ H^2 $ bound of $ \psin{n} $. We first note that \cref{eq:Linfty} now holds for any $ 0 \leq n \leq T/\tau $. Taking $ \alpha=2 $ in \eqref{eq:error_propogation}, using \cref{prop:stability,prop:local_eror} with $ \alpha=2 $ and \cref{eq:claim}, we have for $ 0 \leq n \leq T/\tau - 1 $,
	\begin{eqnarray}\label{eq:e_H2}
		\| \en{n+1} \|_{H^2}
		&&\leq  \| \Psi^\tau (\psi(t_n)) - \Phi^\tau (\psi(t_n))  \|_{H^2} + \| \Phi^\tau (\psi(t_n)) - \Phi^\tau(\psin{n}) \|_{H^2} \nonumber\\
		&&\leq \| \en{n} \|_{H^2} + C_0 \| \en{n} \|_{L^2} + C_1 \tau \nonumber\\
		&&\leq \| \en{n} \|_{H^2} + C_0 e^{C_0T}C_1 \tau + C_1 \tau.
	\end{eqnarray}
	Summing \eqref{eq:e_H2} from $ 0 $ to $ n-1 $, we obtain
	\begin{equation}\label{uniformH2bound}
		\| \en{n} \|_{H^2} \leq C_0 e^{C_0T}C_1 n \tau + C_1 n \tau \leq 2C_0e^{C_0T} C_1 T, \quad 0 \leq n \leq T/\tau.
	\end{equation}
	
	Finally, combining \cref{eq:claim,uniformH2bound}, and using the interpolation inequality for the $ H^1 $-error bound, we prove \cref{eq:semi_error_L2}.
\end{proof}

\subsection{Proof of the optimal $ H^1 $-error bound \cref{eq:semi_error_H1}}
To prove \cref{eq:semi_error_H1}, we assume that $ V \in W^{1, 4}(\Omega) \cap H^1_\text{per}(\Omega) $, $ f $ satisfies Assumption \cref{B}, $ \psi \in C([0, T]; H_\text{\rm per}^3(\Omega)) \cap C^1([0, T]; H^1(\Omega)) $ and $ 0<\tau <\tau_0 $ with $ \tau_0 $ given in \cref{eq:tau_0}. Under the assumptions above, $ B:H^1_\text{per}(\Omega) \rightarrow H^1_\text{per}(\Omega) $ satisfies
\begin{equation}\label{eq:diff_B_H1}
	\| B(v) - B(w) \|_{H^1} \leq C(\| v \|_{H^3}, \| w \|_{H^2}, \| V \|_{W^{1, 4}(\Omega)}) \| v-w \|_{H^1}, \quad v, w \in H^2(\Omega). \vspace{-1.5em}
\end{equation}
\begin{proof}[Proof of \cref{eq:semi_error_H1} in \cref{thm:error_estimates}]
	From \eqref{eq:local_error_integral}, using \cref{eq:diff_B_H1} and the isometry property of $ e^{i t \Delta} $, and noting that $ \psi \in C^1([0, T]; H^1(\Omega)) $, we have, for $ 0 \leq n \leq T/\tau-1 $,
	\begin{equation}\label{eq:semi_local_error_H1}
		\| \psi(t_{n+1}) - \Phi^\tau(\psi(t_n)) \|_{H^1} \leq \int_0^\tau \| B(\psi(t_n+s)) - B(\psi(t_n)) \|_{H^1} \rmd s \lesssim \tau^2.
	\end{equation}
	Noting that $ |\vphi_1(i \theta)| \leq 1 $ for $ \theta \in \R $, we have
	\begin{equation}\label{eq:phi_1_H1}
		\| \vphi_1(i \tau \Delta) v\|_{H^1} \leq \| v \|_{H^1}, \quad v \in H^1_\text{per}(\Omega),
	\end{equation}
	which implies, by recalling \cref{eq:Phi} and using \cref{eq:diff_B_H1} again,
	\begin{eqnarray}\label{eq:semi_stability_H1}
		\| \Phi^\tau(\psi(t_n)) - \Phi^\tau(\psin{n}) \|_{H^1}
		&&\leq \| \psi(t_n) - \psin{n} \|_{H^1} + \tau \| B(\psi(t_n)) - B(\psin{n}) \|_{H^1} \nonumber\\
		&&\leq (1 + C \tau) \| \psi(t_n) - \psin{n} \|_{H^1}, \quad  0 \leq n \leq T/\tau-1,
	\end{eqnarray}
	where $ C $ depends on $ \| V \|_{W^{1, 4}} $, $ \| \psi(t_n) \|_{H^3} $ and $ \| \psin{n} \|_{H^2} $, which are uniformly bounded. Then \cref{eq:semi_error_H1} follows from \eqref{eq:semi_local_error_H1} and \eqref{eq:semi_stability_H1} by the standard Lady Windermere's fan argument.
\end{proof}

\section{Optimal error bounds for the full discretization \eqref{EWI-FS}}
In this section, we establish optimal error bounds in $L^2$- and $H^1$-norm for the full-discretization scheme EWI-FS \eqref{EWI-FS}, and generalize them to the EWI-EFP scheme \cref{EWI-EFP}.

\subsection{Main results}
For $ \psihn{n}\ (0 \leq n \leq T/\tau) $ obtained by the EWI-FS scheme \cref{EWI-FS}, we have
\begin{theorem}\label{thm:error_estimates_FS}
	Assume that $ V \in L^\infty(\Omega) $, $ f $ satisfies Assumption \cref{A} and the exact solution $ \psi \in C([0, T]; H_{\text{\rm per}}^2(\Omega)) \cap C^1([0, T]; L^2(\Omega)) $, there exists $ \tau_0>0 $ and $ h_0>0 $ depending on $ M $ and $ T $ and sufficiently small such that for any $ 0 < \tau < \tau_0 $ and $ 0 < h < h_0 $, we have
	\begin{equation}\label{eq:full_error_L2}
		\begin{aligned}
			&\| \psi(\cdot, t_n) - \psihn{n} \|_{L^2} \lesssim \tau + h^2, \quad \| \psihn{n} \|_{H^2} \leq C(M), \\
			&\| \psi(\cdot, t_n) - \psihn{n} \|_{H^1} \lesssim \sqrt{\tau} + h,  \qquad 0 \leq n \leq T/\tau.
		\end{aligned}
	\end{equation}
	Moreover, if $ V \in W^{1, 4}(\Omega) \cap H_\text{\rm per}^1(\Omega) $, $ f $ satisfies \cref{B} and $ \psi \in C([0, T]; H_\text{\rm per}^3(\Omega)) \cap C^1([0, T]; H^1(\Omega)) $, we have, for $ 0 < \tau < \tau_0 $ and $ 0 < h < h_0 $,
	\begin{equation}\label{eq:full_error_H1}
		\| \psi(\cdot, t_n) - \psihn{n} \|_{L^2} \lesssim \tau + h^3, \quad \| \psi(\cdot, t_n) - \psihn{n} \|_{H^1} \lesssim \tau + h^2, \qquad 0 \leq n \leq T/\tau.
	\end{equation}
\end{theorem}

\begin{remark}\label{rem:coupling}
	Thanks to the strong $ H^2 $-control of the semi-discretization solution in \cref{eq:semi_error_L2}, there is no coupling condition between $ \tau $ and $ h $ for all $ 1 \leq d \leq 3 $ in \cref{thm:error_estimates_FS}.
\end{remark}


In the following, we shall prove \cref{thm:error_estimates_FS}. We use different methods to prove \cref{eq:full_error_L2} and \cref{eq:full_error_H1}. For the $ L^2 $-norm error bound \cref{eq:full_error_L2}, we compare the full-discretization solution $ \psihn{n} $ with the semi-discretization solution $ \psin{n} $ to avoid the coupling condition between $ \tau $ and $ h $ when using the inverse inequalities. Then, for the $ H^1 $-norm error bound \cref{eq:full_error_H1}, we can directly compare the full-discretization solution with the exact solution since we already have control of the full-discretization solution in $ H^2 $-norm.

In the rest of this section, we assume that $ V \in L^\infty(\Omega) $, $ f $ satisfies Assumption \cref{A} and $ \psi \in C([0, T]; H_{\text{\rm per}}^2(\Omega)) \cap C^1([0, T]; L^2(\Omega)) $.

\subsection{Proof of the optimal $ L^2 $-error bound \cref{eq:full_error_L2}}
We start with the error estimates between the semi-discretization solution $ \psin{n} $ and the full-discretization solution $ \psihn{n} $.

\begin{proposition}\label{thm:truncation_error}
	Let $ 0<\tau<\tau_0 $ with $ \tau_0 $ given in \cref{thm:error_estimates}. Then there exists $ h_0 $ depending on $ M $ and $ T $ and small enough such that for $ 0<h<h_0 $, we have
	\begin{equation*}
		\| P_N \psin{n} - \psihn{n} \|_{L^2} \leq C(M, T)h^2, \quad 0 \leq n \leq T/\tau.
	\end{equation*}
\end{proposition}

\begin{proof}
	Define the error function $ \ehn{n} := P_N \psin{n} - \psihn{n} $ for $ 0 \leq n \leq T/\tau $. Then $ \ehn{0} = P_N \psin{0} - \psihn{0} = 0 $. Applying $ P_N $ on both sides of \cref{EWI}, noting that $ P_N $ commutes with $ e^{i \tau \Delta} $ and $ \vphi_1(i \tau \Delta) $ and recalling \cref{eq:B_def}, we have
	\begin{equation}\label{eq:P_Npsi}
		P_N \psin{n+1} = e^{i \tau \Delta} P_N \psin{n} - i \tau \vphi_1(i \tau \Delta) P_N B(\psin{n}), \quad 0 \leq n \leq T/\tau-1.
	\end{equation}
	Recalling \cref{EWI-FS,eq:B_def}, we have
	\begin{equation}\label{eq:psi_h}
		\psihn{n+1} = e^{i \tau \Delta} \psihn{n} - i \tau \vphi_1(i \tau \Delta) P_N B(\psihn{n}), \quad 0 \leq n \leq T/\tau-1.
	\end{equation}
	Subtracting \cref{eq:psi_h} from \cref{eq:P_Npsi}, we have, for $0 \leq n \leq T/\tau-1$, 
	\begin{equation}\label{eq:eh_equation}
		\ehn{n+1} = e^{i \tau \Delta} \ehn{n} - i \tau \vphi_1(i \tau \Delta) P_N (B(\psin{n}) - B(\psihn{n})).
	\end{equation}
	From \cref{eq:eh_equation}, using the isometry property of $ e^{i t \Delta} $, the $ L^2 $-projection property of $ P_N $ and \cref{lem:smoothingofphi1} with $ \alpha=0 $, we have, for $ 0 \leq n \leq T/\tau - 1 $,
	\begin{eqnarray}\label{eq:errorh}
			\| \ehn{n+1} \|_{L^2}
			&\leq&\| \ehn{n} \|_{L^2} + \tau \| \vphi_1(i \tau \Delta) P_N ((B(\psin{n}) - B(\psihn{n})) \|_{L^2} \nonumber\\
			&\leq&\| \ehn{n} \|_{L^2} + \tau \| (B(\psin{n}) - B(\psihn{n}) \|_{L^2}\nonumber \\
			&\leq&\| \ehn{n} \|_{L^2} + \tau \| (B(\psin{n}) - B(P_N \psin{n}) \|_{L^2} + \tau \| (B(P_N \psin{n}) - B(\psihn{n}) \|_{L^2}.
	\end{eqnarray}
	By \cref{eq:semi_error_L2}, using Sobolev embedding and the boundedness of $ P_N $, we have
	\begin{equation}\label{eq:linfty_psin}
		\| P_N \psin{n} \|_{L^\infty} \leq \tilde c \| P_N \psin{n} \|_{H^2} \leq \tilde c \| \psin{n} \|_{H^2} \leq \tilde c C(M) =: M_0, \quad 0 \leq n \leq T/\tau, 
	\end{equation}
	where $\tilde{c}$ is given by the Sobolev embedding $H^2 \hookrightarrow L^\infty$. Similarly, $ \| \psin{n} \|_{L^\infty} \leq M_0 $. From \eqref{eq:errorh}, noting \cref{eq:linfty_psin}, using \cref{lem:diff_B}, the uniform $H^2$-bound in \cref{eq:semi_error_L2}, and the standard projection error estimate $ \| \phi - P_N \phi \|_{L^2} \lesssim h^2 | \phi |_{H^2} \  \forall \phi \in H^2_\text{per}(\Omega) $, we have, for $ 0 \leq n \leq T/\tau - 1 $,
	\begin{equation}\label{eq:error_propagation_full}
		\| \ehn{n+1} \|_{L^2} \leq \| \ehn{n} \|_{L^2} + \CL(\max\{M_0, \| \psihn{n} \|_{L^\infty}\}) \tau \| \ehn{n} \|_{L^2} + \tilde C_1 \tau h^2, 
	\end{equation}
	where $\tilde C_1$ depends exclusively on $M$. The conclusion then follows from the discrete Gronwall's inequality and the standard induction argument by using the inverse inequality \cite{book_spectral} 
	\begin{equation}\label{eq:inverse}
		\| \phi \|_{L^\infty} \leq C_\text{inv} h^{-\frac{d}{2}}\| \phi \|_{L^2}, \quad \phi \in X_N, 
	\end{equation}
	where $ d $ is the dimension of the space, i.e. $ d=1 $ in the current case. For the convenience of the reader, we present this process in the following. 
	
	{Let $ \tilde C_0 := \CL(1+M_0) $ with $\CL(\cdot)$ given in \cref{lem:diff_B} and recall $\tilde C_1$ given by \cref{eq:error_propagation_full}. Let $ 0<h_0<1 $ be chosen such that 
	\begin{equation}\label{eq:h_0}
			C_\text{inv}e^{\tilde C_0 T} \tilde C_1 h_0^{2-d/2} \leq 1. 
		\end{equation}
	We shall show that, when $0<h<h_0$, for $0 \leq n \leq T/\tau$, 
	\begin{equation}\label{eq:claim_2}
		\| \ehn{n} \|_{L^2} \leq e^{\tilde C_0 T} \tilde C_1 h^2, \quad \| \psihn{n} \|_{L^\infty} \leq 1+M_0. 
	\end{equation}
	Recall that $ \ehn{0} = 0 $, and by \cref{eq:linfty_psin}, $ \| \psihn{0} \|_{L^\infty}=\| P_N \psi_0 \|_{L^\infty} = \| P_N \psin{0} \|_{L^\infty} \leq M_0 $. Then \cref{eq:claim_2} holds for $n=0$. Assume that \cref{eq:claim_2} holds for $ 0 \leq n \leq m \leq T/\tau-1 $, which implies, from \cref{eq:error_propagation_full},  
	\begin{equation}\label{eq:error_propagation_induction}
		\| \ehn{n+1} \|_{L^2} \leq (1 + \tilde C_0 \tau) \| \ehn{n} \|_{L^2} + \tilde C_1 \tau h^2, \quad 0 \leq n \leq m,  
	\end{equation}
	which further implies, by using discrete Gronwall's inequality, 
	\begin{equation}\label{eq:eh_L2}
		\| \ehn{m+1} \|_{L^2} \leq e^{\tilde C_0 T} \tilde C_1 h^2. 
	\end{equation}
	From \cref{eq:eh_L2}, using \cref{eq:inverse}, recalling \cref{eq:linfty_psin,eq:h_0}, we have, by triangle inequality, 
	\begin{equation}\label{eq:Linfty_psi_h}
		\begin{aligned}
			\| \psihn{m+1} \|_{L^\infty} 
			&\leq \| \ehn{m+1} \|_{{L^\infty}} + \| P_N \psin{m+1} \|_{L^\infty} \leq C_\text{inv} h^{2-d/2} \| \ehn{m+1} \|_{{L^2}} + M_0 \\
			&\leq C_\text{inv} e^{\tilde C_0 T} \tilde C_1 {h^{2-d/2}} + M_0 \leq 1+M_0. 
		\end{aligned}
	\end{equation}
	Combing \cref{eq:eh_L2,eq:Linfty_psi_h} proves \cref{eq:claim_2} for $ n = m+1 $, and thus for all $ 0 \leq n \leq T/\tau $ by mathematical induction, which completes the proof.}
\end{proof}

\begin{proof}[Proof of \cref{eq:full_error_L2} in \cref{thm:error_estimates_FS}]
	By triangle inequality, for $ 0 \leq \alpha \leq 2 $,
	\begin{equation}\label{eq:decomp_full}
		\|  \psi(t_n) - \psihn{n} \|_{H^\alpha} \leq \| \psi(t_n) - \psin{n}  \|_{H^\alpha} + \| \psin{n} - P_N \psin{n} \|_{H^\alpha} + \| P_N \psin{n} - \psihn{n} \|_{H^\alpha}.
	\end{equation}
	From \cref{eq:semi_error_L2}, using the interpolation inequalities, we have
	\begin{equation}\label{eq:semi_error_Hs}
		 \| \psi(t_n) - \psin{n}  \|_{H^\alpha} \lesssim \tau^{1-\alpha/2}, \quad 0 \leq \alpha \leq 2.
	\end{equation}
	From \cref{eq:decomp_full}, using \cref{eq:semi_error_Hs} and the standard Fourier projection error estimates
	\begin{equation}
		\| \phi - P_N \phi \|_{H^\alpha} \lesssim h^{2-\alpha} | \phi |_{H^2}, \quad 0 \leq \alpha \leq 2, \quad \phi \in H^2_{\text{per}}(\Omega),
	\end{equation}
	and noting \cref{eq:semi_error_L2}, we have
	\begin{equation}\label{eq:error_decomp_full}
		\| \psi(t_n) - \psihn{n}  \|_{H^\alpha} \lesssim \tau^{1-\alpha/2} + h^{2-\alpha} + \|  P_N \psin{n} - \psihn{n} \|_{H^\alpha}.
	\end{equation}
	From \cref{eq:error_decomp_full}, using the inverse estimate $ \| \phi \|_{H^\alpha} \lesssim h^{-\alpha} \| \phi \|_{L^2} \  \forall \phi \in X_N $ \cite{book_spectral_2,book_spectral} and \cref{thm:truncation_error}, we have
	\begin{equation}\label{eq:full_error_proof}
		\| \psihn{n} - \psi(t_n) \|_{H^\alpha} \lesssim \tau^{1-\alpha/2} + h^{2-\alpha} + h^{-\alpha}h^2 \lesssim \tau^{1-\alpha/2} + h^{2-\alpha}, \quad 0 \leq \alpha \leq 2,
	\end{equation}
	which proves \cref{eq:full_error_L2} by taking $ \alpha=0, 1, 2 $.
\end{proof}

\subsection{Proof of the  optimal $ H^1 $-error bound \eqref{eq:full_error_H1}}
To prove \cref{eq:full_error_H1}, we assume that $ V \in W^{1, 4}(\Omega) \cap H_\text{\rm per}^1(\Omega) $, $ f $ satisfies Assumption \cref{B}, $ \psi \in C([0, T]; H_\text{\rm per}^3(\Omega)) \cap C^1([0, T]; H^1(\Omega)) $ and $ 0<\tau<\tau_0 $, $ 0<h<h_0 $. Since we already have the uniform control of $ \psihn{n} $ in $ H^2 $-norm, \cref{eq:full_error_H1} can be proved similarly to \cref{eq:semi_error_H1}, and we just sketch the outline here.

\begin{proof}[Proof of \cref{eq:full_error_H1} in \cref{thm:error_estimates_FS}]
	Recalling \cref{eq:duhamel,EWI-FS}, we have
	\begin{eqnarray}
		&&P_N \psi(t_{n+1}) - \Phih^\tau(P_N \psi(t_n)) \nonumber \\
		&&=-i \int_0^\tau e^{i (\tau - s)\Delta} P_N \left ( B(\psi(t_n+s)) - B(P_N \psi(t_n)) \right ) \rmd s, \quad 0 \leq n \leq \frac{T}{\tau}-1,
	\end{eqnarray}
	which implies, by the property of $ e^{it\Delta} $ and $ P_N $, for $ X=L^2 $ or $ H^1 $, that
	\begin{eqnarray}
		&&\| P_N \psi(t_{n+1}) - \Phih^\tau(P_N \psi(t_n)) \|_{X} \nonumber \\
		&&\leq \int_0^\tau \left(\| B(\psi(t_n+s)) - B(\psi(t_n)) \|_{X} + \| B(\psi(t_n))- B(P_N \psi(t_n)) \|_{X} \right) \rmd s. \label{local_error_decomp_full}
	\end{eqnarray}
	From \eqref{local_error_decomp_full}, using \cref{lem:diff_B} and \cref{eq:diff_B_H1}, we have, for $ 0 \leq n \leq T/\tau-1$,
	\begin{equation}
		\begin{aligned}
			&\| P_N \psi(t_{n+1}) - \Phih^\tau(P_N \psi(t_n)) \|_{L^2} \lesssim \tau^2 + \tau h^3, \\
			&\| P_N \psi(t_{n+1}) - \Phih^\tau(P_N \psi(t_n)) \|_{H^1} \lesssim \tau^2 + \tau h^2.
		\end{aligned}
	\end{equation}
	Besides, recalling \cref{EWI-FS}, using \cref{lem:smoothingofphi1,lem:diff_B}, \cref{eq:smoothingofphi1} and \cref{eq:diff_B_H1}, we have
	\begin{equation}
		\begin{aligned}
			&\| \Phih^\tau(P_N \psi(t_n)) - \Phih^\tau(\psihn{n}) \|_{L^2} \leq (1+C_1\tau) \| P_N \psi(t_n) - \psihn{n} \|_{L^2}, \\
			&\| \Phih^\tau(P_N \psi(t_n)) - \Phih^\tau(\psihn{n}) \|_{H^1} \leq (1+C_2\tau) \| P_N \psi(t_n) - \psihn{n} \|_{H^1},
		\end{aligned}
	\end{equation}
	where $ C_1 $ depends on $ \| P_N \psi(t_n) \|_{L^\infty} $ and $ \| \psihn{n} \|_{L^\infty} $, and $ C_2 $ depends on $ \| P_N \psi(t_n) \|_{H^3} $ and $ \| \psihn{n} \|_{H^2} $ for $ 0 \leq n \leq T/\tau -1 $. Thus, both $ C_1 $ and $ C_2 $ are under control. Then the proof can be completed by the Lady Windermere's fan argument and the standard projection error estimates of $ P_N $.
\end{proof}

\subsection{Extension to the EWI-EFP \cref{EWI-EFP}}
{For $ I_N \psihhn{n}(0 \leq n \leq T/\tau) $ obtained from the EWI-EFP scheme \cref{EWI-EFP}, it satisfies the same error bounds as $\psihn{n} (0 \leq n \leq T/\tau)$ in \Cref{thm:error_estimates_FS}, under the same assumptions on potential and the exact solution, but with a little more regular nonlineairty.} To be precise, we introduce another assumption on the nonlinearity as
\begin{equation}\label{C}
	f(|v|^2)v \in H^\alpha_\text{per}(\Omega), \quad \forall v \in H^\alpha_\text{per}(\Omega). \tag{C}
\end{equation}

For the optimal $L^2$-norm error bound, we assume that $ f $ satisfies Assumptions \cref{A} and \cref{C} with $\alpha=2$. Two typical examples of $f$ include (i) $ f(\rho) = \lambda_1\rho^{\sigma_1} + \lambda_2\rho^{\sigma_2} $ with $ \sigma_2>\sigma_1 \geq 1/2 $ and $\lambda_1,\lambda_2\in{\mathbb R}$, and (ii) $ f(\rho) =\lambda \rho^\sigma \ln \rho $ with $ \sigma > 1/2 $ and $\lambda\in{\mathbb R}$.

For the optimal $H^1$-norm error bound, we assume, in addition to Assumption \cref{B}, $ f $ satisfies \cref{C} with $ \alpha=3 $ and the discrete counterpart of Assumption \cref{B}
\begin{equation}\label{B'}
	\| I_N (f(|v|^2)v - f(|w|^2)w)  \|_{H^1} \leq C(\| v \|_{H^3}, \| w \|_{H^2}) \| v-w \|_{H^1}, \quad v, w \in X_N, \tag{B'}
\end{equation}
with two typical examples of $f$: (i) $ f(\rho) =\lambda_1 \rho^{\sigma_1} +\lambda_2 \rho^{\sigma_2} $ with $ \sigma_2>\sigma_1 \geq 1 $ and $\lambda_1,\lambda_2\in{\mathbb R}$, and (ii) $ f(\rho) = \lambda\rho^\sigma \ln \rho $ with $ \sigma > 1 $ and $\lambda\in{\mathbb R}$ (see, e.g., \cite{bao2023} for the proof).

Then we have the following error bounds for the EWI-EFP scheme \cref{EWI-EFP}.
\begin{corollary}\label{coro:error_estimates_EFP}
	Assume that $ V \in L^\infty(\Omega) $, $ f $ satisfies Assumptions \cref{A} and \cref{C} with $\alpha=2$ and the exact solution $ \psi \in C([0, T]; H_{\text{\rm per}}^2(\Omega)) \cap C^1([0, T]; L^2(\Omega)) $. There exists $ \tau_0>0 $ and $ h_0>0 $ depending on $ M $ and $ T $ and sufficiently small such that for any $ 0 < \tau < \tau_0 $ and $ 0 < h < h_0 $, we have
	\begin{equation}\label{eq:full_error_L2_EFP}
		\begin{aligned}
			&\| \psi(\cdot, t_n) - I_N \psihhn{n} \|_{L^2} \lesssim \tau + h^2, \quad \| I_N \psihhn{n} \|_{H^2} \leq C(M), \\
			&\| \psi(\cdot, t_n) - I_N \psihhn{n} \|_{H^1} \lesssim \sqrt{\tau} + h,  \qquad 0 \leq n \leq T/\tau.
		\end{aligned}
	\end{equation}
	Moreover, if $ V \in W^{1, 4}(\Omega) \cap H_\text{\rm per}^1(\Omega) $, $ f $ satisfies Assumptions \cref{B}, \cref{B'} and \cref{C} with $\alpha=3$ and $ \psi \in C([0, T]; H_\text{\rm per}^3(\Omega)) \cap C^1([0, T]; H^1(\Omega)) $, we have, for $ 0 < \tau < \tau_0 $ and $ 0 < h < h_0 $,
	\begin{equation}\label{eq:full_error_H1_EFP}
		\| \psi(\cdot, t_n) - I_N \psihhn{n} \|_{L^2} \lesssim \tau + h^3, \  \| \psi(\cdot, t_n) - I_N \psihhn{n} \|_{H^1} \lesssim \tau + h^2, \quad 0 \leq n \leq T/\tau.
	\end{equation}
\end{corollary}

{
For notational simplicity, we define $B_N:C(\overline{\Omega}) \rightarrow X_N$ as
\begin{equation}\label{eq:BN_def}
	B_N(\phi) := P_N(V \phi) + I_N \F(\phi), \quad \phi \in C(\overline{\Omega}), 
\end{equation}
where $\F(\phi)(x)=f(|\phi(x)|^2)\phi(x)$ for $x \in \Omega$. Then we have
\begin{lemma}\label{lem:PNB-BN}
	Let $v, w \in X_N$. Assume that $V \in L^\infty(\Omega)$ and $f$ satisfies \cref{A} and \cref{C} with $\alpha=2$. If $\| v \|_{H^2} \leq M_0$ and $\| w \|_{L^\infty} \leq M_1$, we have
	\begin{equation}\label{eq:PNB-BN}
		\| P_N B(v) - B_N(w) \|_{L^2} \leq C(\|V\|_{L^\infty}, M_0, M_1)\| v-w \|_{L^2} + C(M_0)h^2.  
	\end{equation}
	Moreover, if $V \in W^{1, 4}(\Omega) \cap H^1_\text{per}(\Omega)$, $f$ satisfies Assumptions \cref{B'} and \cref{C} with $\alpha = 3$, and $\| v \|_{H^3} \leq M_2$, and $\| w \|_{H^2} \leq M_3$, we have
	\begin{equation}\label{eq:PNB-BN_2}
		\| P_N B(v) - B_N(w) \|_{H^1} \leq C(\| V \|_{W^{1, 4}}, M_2, M_3)\| v-w \|_{H^1} + C(M_2)h^2. 
	\end{equation}
\end{lemma}

\begin{proof}
	Recalling \cref{eq:B_def,eq:BN_def}, we have
	\begin{align}\label{eq:PNB-BN_eq}
		P_N B(v) - B_N(w) 
		&= P_N(V(v - w)) + P_N \F(v) - I_N \F(w) \notag \\
		&= P_N(V(v - w)) + (P_N - I_N) \F(v) + I_N \F(v) - I_N \F(w). 
	\end{align}
	From \cref{eq:PNB-BN_eq}, using assumption \cref{C} with $\alpha = 2$, we have
	\begin{align}\label{eq:PNB-BN_est_1}
		\| P_N B(v) - B_N(w) \|_{L^2} 
		&\lesssim \| V \|_{L^\infty} \| v-w \|_{L^2} + C(M_0) h^2 + \| I_N \F(v) - I_N \F(w) \|_{L^2} \notag \\
		&\leq \| V \|_{L^\infty} \| v-w \|_{L^2} + C(M_0) h^2 + C(M_0, M_1) \| v - w \|_{L^2}, 
	\end{align}
	where we use $\| I_N \F(\phi) \|_{L^2}^2 = h\sum_{j \in \mathcal{T}_N^0} |\F(\phi(x_j))|^2$ and $I_N$ is an identity on $X_N$ in the last line, and we prove \cref{eq:PNB-BN}. 
	
	To prove \cref{eq:PNB-BN_2}, using \cref{C} with $\alpha = 3$ and \cref{B'}, from \cref{eq:PNB-BN_eq}, we have
	\begin{equation*}
		\| P_N B(v) - B_N(w) \|_{H^1} \lesssim  \| V \|_{W^{1, 4}}\| v-w \|_{H^1} + C(M_2) h^2 + C(M_2, M_3)\| v-w \|_{H^1}, 
	\end{equation*}
	which proves \cref{eq:PNB-BN_2}, and completes the proof. 
\end{proof}

\begin{proof}[Proof of \cref{coro:error_estimates_EFP}]
	The proof is similar to the proof of \cref{thm:error_estimates_FS}, and we sketch it here for the convenience of the reader. We start with the proof of \cref{eq:full_error_L2_EFP}. Define the error function $ \ehhn{n} := P_N \psin{n} - I_N \psihhn{n} $ for $ 0 \leq n \leq T/\tau $. Then $ \ehhn{0} = P_N \psi_0 - I_N \psi_0 $ satisfies $\|\ehhn{0}\|_{L^2} \leq C(M) h^2$. Recalling \cref{EWI-FSP_equation,eq:BN_def,eq:P_Npsi}, 
	we obtain, for $0 \leq n \leq T/\tau-1$, 
	\begin{equation}\label{eq:error_eq_EFP}
		\ehhn{n+1} = e^{i\tau\Delta} \ehhn{n} - i \tau \vphi_1(i\tau\Delta) (P_N B (\psin{n}) - B_N(I_N \psihhn{n})). 
	\end{equation}
	From \cref{eq:error_eq_EFP}, by the boundedness of $e^{i\tau\Delta}$, $P_N$ and $\vphi_1(i\tau\Delta)$, \cref{lem:diff_B}, triangle inequality, the uniform $H^2$-bound of $\psin{n}$ in \cref{eq:semi_error_L2}, and \cref{eq:PNB-BN}, we get
	\begin{align}\label{eq:error_eq_EFP_L2}
		\| \ehhn{n+1} \|_{L^2} 
		&\leq \| \ehhn{n} \|_{L^2} + \tau \left( \| P_N B (\psin{n}) - P_NB(P_N \psin{n}) \|_{L^2} \right.  \notag \\
		&\quad + \left. \| P_N B (P_N \psin{n}) - B_N(I_N \psihhn{n}) \|_{L^2} \right) \notag \\
		&\leq \| \ehhn{n} \|_{L^2} + C(M) \tau h^2 + (1+ C(M, \| I_N \psihhn{n} \|_{L^\infty}) \tau) \| \ehhn{n} \|_{L^2}. 
	\end{align}
	From \cref{eq:error_eq_EFP_L2}, by discrete Gronwall's inequality and the same induction process as in the proof of \cref{thm:truncation_error}, noting first step error $\| \ehhn{0} \|_{L^2} \leq C(M)h^2$, we obtain
	\begin{equation}
		\| \ehhn{n} \|_{L^2} \leq C(M, T) h^2, \quad 0 \leq n \leq T/\tau. 
	\end{equation} 
	The rest of the proof of \cref{eq:full_error_L2_EFP} can be completed by following the proof of \cref{eq:full_error_L2}. 
	
	Then we outline the proof of \cref{eq:full_error_H1_EFP}. Define the numerical flow $\Phiht: X_N \rightarrow X_N$ associated with the EWI-EFP scheme \cref{EWI-EFP} as
	\begin{equation}\label{eq:EWI-EFP_flow}
		\Phiht(v) = e^{i\tau \Delta} v - i \tau \vphi_1(i\tau\Delta) B_N(v), \quad v \in X_N. 
	\end{equation}
	Recalling \cref{EWI-FSP_equation}, we have $I_N \psihhn{n+1} = \Phiht(I_N \psihhn{n})$ for $n \geq 0$. Recalling \cref{eq:duhamel,eq:EWI-EFP_flow}, the local truncation error can be decomposed as
	\begin{align}
		P_N \psi(t_{n+1}) - \Phiht(P_N \psi(t_n)) 
		= - i \int_0^\tau e^{i(\tau - s)\Delta} \big( P_N B(\psi(t_n + s))- P_N B(\psi(t_n)) + \notag \\
		P_N B(\psi(t_n)) - P_N B(P_N \psi(t_n)) + P_N B(P_N \psi(t_n)) - B_N(P_N \psi(t_n)) \big) \rmd s, 
	\end{align}
	which implies, by the boundedness of $e^{it\Delta}$ and $P_N$, and using \cref{eq:diff_B_H1,eq:PNB-BN_2},  
	\begin{equation}\label{eq:local_error_EFP_H1}
		\| P_N \psi(t_{n+1}) - \Phiht(P_N \psi(t_n)) \|_{H^1} \lesssim \tau^2 + \tau h^2. 
	\end{equation}
	Besides, recalling \cref{eq:EWI-EFP_flow} and using \cref{B'}, we have $H^1$-stability estimate
	\begin{equation}\label{eq:stability_EFP_H1}
		\| \Phiht(P_N \psi(t_n)) - \Phiht(I_N \psihhn{n}) \|_{H^1} \leq (1 + C_3 \tau) \| P_N \psi(t_n) - I_N \psihhn{n} \|_{H^1}, 
	\end{equation}
	where $C_3$ depends on $\| V \|_{W^{1,4}}$, $\| \psi(t_n) \|_{H^3}$ and $\| I_N \psihhn{n} \|_{H^2}$, and thus is under control. The proof of the $H^1$-error bound in \cref{eq:full_error_H1_EFP} can be completed by applying standard Lady Windermere's fan argument with \cref{eq:local_error_EFP_H1,eq:stability_EFP_H1}. The proof of the $L^2$-error bound in \cref{eq:full_error_H1_EFP} can be obtained similarly. Then the proof is completed. 
\end{proof}

}
\section{Numerical results}
In this section, we present some numerical examples for the NLSE with either low regularity potential or nonlinearity. In the following, we fix $ \Omega = (-16, 16) $, $ T=1 $, $ d=1 $ and consider the power-type nonlinearity $ f(\rho) = -\rho^\sigma\ (\sigma > 0) $.

Let $ \psihn{n} (0 \leq n \leq T/\tau) $ be the numerical solution obtained by the EWI-FS method \cref{EWI-FS} or the EWI-EFP method \cref{EWI-EFP}, which will be made clear in each case. Define the error functions
\begin{equation*}
	e^k_{L^2} = \| \psi(t_k) - I_N \psihn{k}  \|_{L^2}, \quad e^k_{H^1} = \| \psi(t_k) - I_N \psihn{k} \|_{H^1}, \quad 0 \leq k \leq n:=T/\tau.
\end{equation*}

\subsection{For the NLSE with locally Lipschitz nonlinearity}
In this subsection, we only consider the NLSE with the power-type nonlinearity and without potential:
\begin{equation}\label{eq:NLSE_semi-smooth}
	i \partial_t \psi(x, t) = -\Delta \psi(x, t) - |\psi(x, t)|^{2\sigma} \psi(x, t), \quad x \in \Omega, \quad t>0,
\end{equation}
where $ \sigma > 0 $. Recall that Assumption \cref{A} is satisfied for any $ \sigma > 0 $ and Assumption \cref{B} is satisfied for any $ \sigma \geq 1/2 $. Note that when there is no potential, the extended Fourier pseudospectral method collapses to the standard Fourier pseudospectral method.

Two types of initial data are considered:

(i) Type I. The $ H^2 $ initial datum
	\begin{equation}\label{typeI_ini}
		\psi_0(x) = x|x|^{0.51} e^{-\frac{x^2}{2}}, \quad x \in \Omega.
	\end{equation}
	
(ii) Type II. The smooth initial datum
	\begin{equation}\label{typeII_ini}
		\psi_0(x) = x e^{-\frac{x^2}{2}}, \quad x \in \Omega.
	\end{equation}

The two initial data are specially chosen to demonstrate the influence of the low regularity of $ f $ around the origin. Since both Type I and II initial data are odd functions, the corresponding solutions of the NLSE will satisfy $ \psi(0, t) \equiv 0 $ for all $ t \geq 0 $. The difference of these two initial data lies in the regularity.

We shall test the convergence order in both time and space for Type I and II initial data. For each initial datum, we choose $ \sigma = 0.1, 0.2, 0.3, 0.4 $. The 'exact' solutions are computed by the Strang splitting Fourier pseudospectral method with $ \tau = \tau_\text{e} := 10^{-6} $ and $ h = h_\text{e} := 2^{-9} $. {When test the spatial errors, we fix the time step size $\tau = \tau_e$, and when test the temporal errors, we fix the mesh size $h = h_e$.}

We start with the Type I $ H^2 $ initial datum \cref{typeI_ini}. \cref{fig:comp_FS_FP_semi_smooth_H2_ini} exhibits the spatial error in $ L^2 $- and $ H^1 $-norm of the EWI-FS (solid lines) and the EWI-EFP (dotted lines) method for $ \sigma = 0.1 $ with the Type I initial datum. We can observe that the EWI-FS method is second order convergent in $ L^2 $-norm and first order convergent in $ H^1 $-norm. Moreover, we see that there is almost no difference between the spatial error of the EWI-FS method and the EWI-EFP method, which suggests that the Fourier pseudospectral method seems also suitable to discretize the low regularity nonlinearity.

\cref{fig:conv_dt_semi_smooth_H2_ini} plots the temporal error in $ L^2 $- and $ H^1 $-norm of the EWI for different $ 0<\sigma<1/2 $ with Type I initial datum. \cref{fig:conv_dt_semi_smooth_H2_ini} (a) shows that the temporal convergence is first order in $ L^2 $-norm for all the four $ \sigma $ and \cref{fig:conv_dt_semi_smooth_H2_ini} (b) shows the temporal convergence is half order in $ H^1 $-norm for all the four $ \sigma $.

The results in \cref{fig:comp_FS_FP_semi_smooth_H2_ini,fig:conv_dt_semi_smooth_H2_ini} confirm our optimal $ L^2 $-norm error bound for the NLSE with locally Lipschitz nonlinearity, and demonstrate that it is sharp.

\begin{figure}[htbp]
	\centering
	\includegraphics[width=0.475\textwidth]{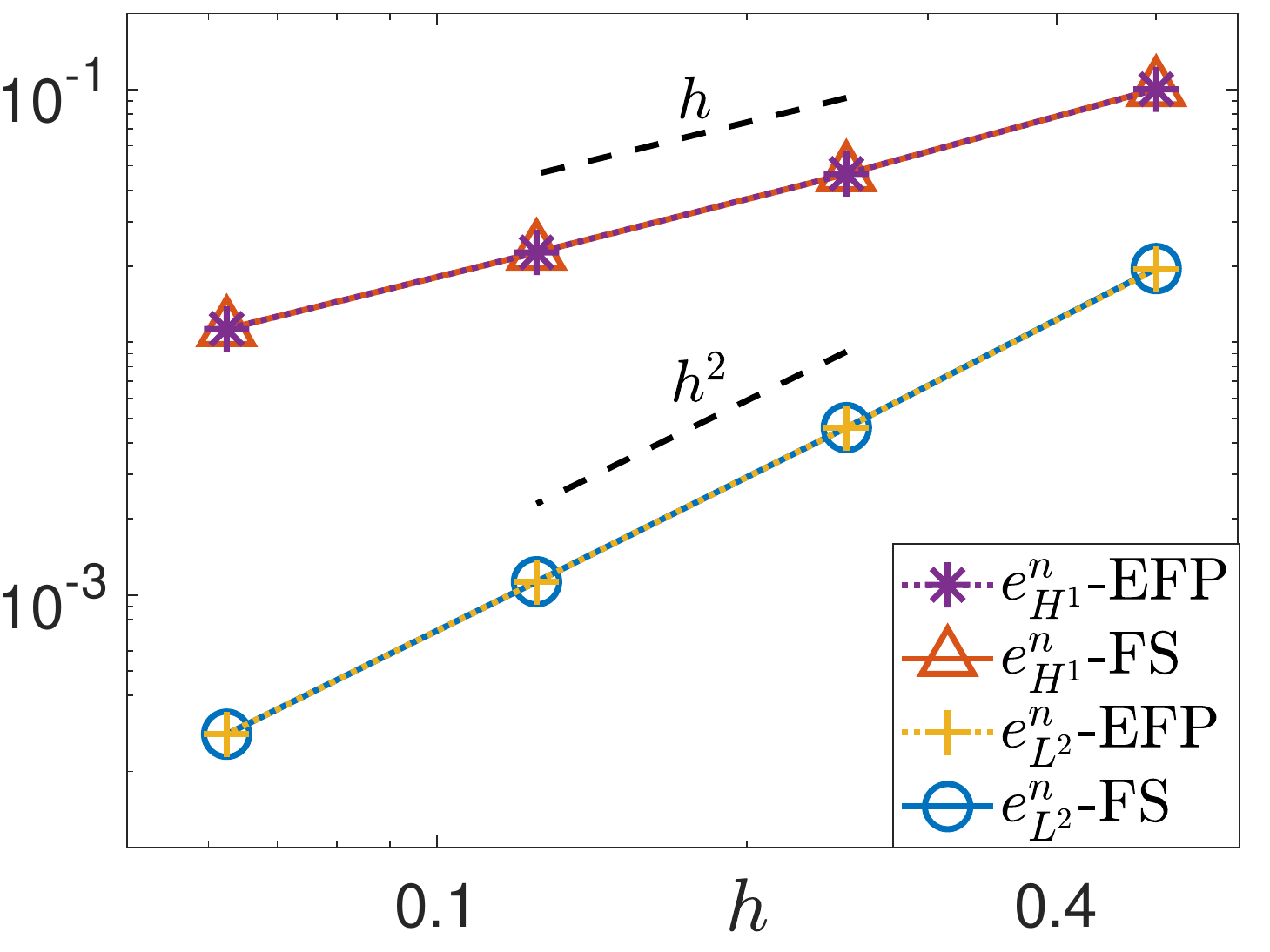}
	\caption{Comparison of the Fourier spectral and pseudospectral discretization of the nonlinear term in \cref{eq:NLSE_semi-smooth} with $ \sigma = 0.1 $ and Type I initial datum \cref{typeI_ini}.}
	\label{fig:comp_FS_FP_semi_smooth_H2_ini}
\end{figure}

\begin{figure}[htbp]
	\centering
	\includegraphics[width=0.475\textwidth]{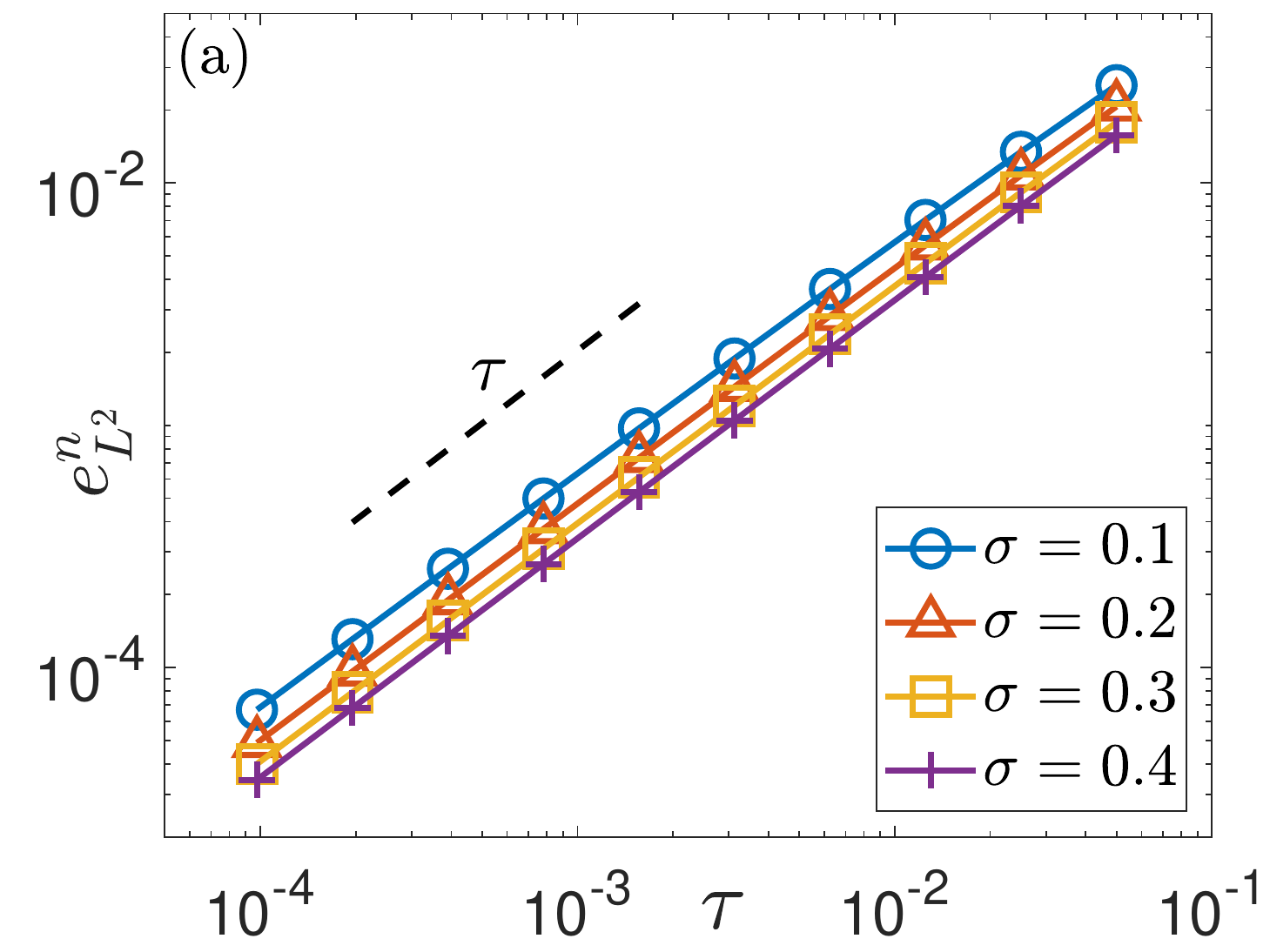}
	\includegraphics[width=0.475\textwidth]{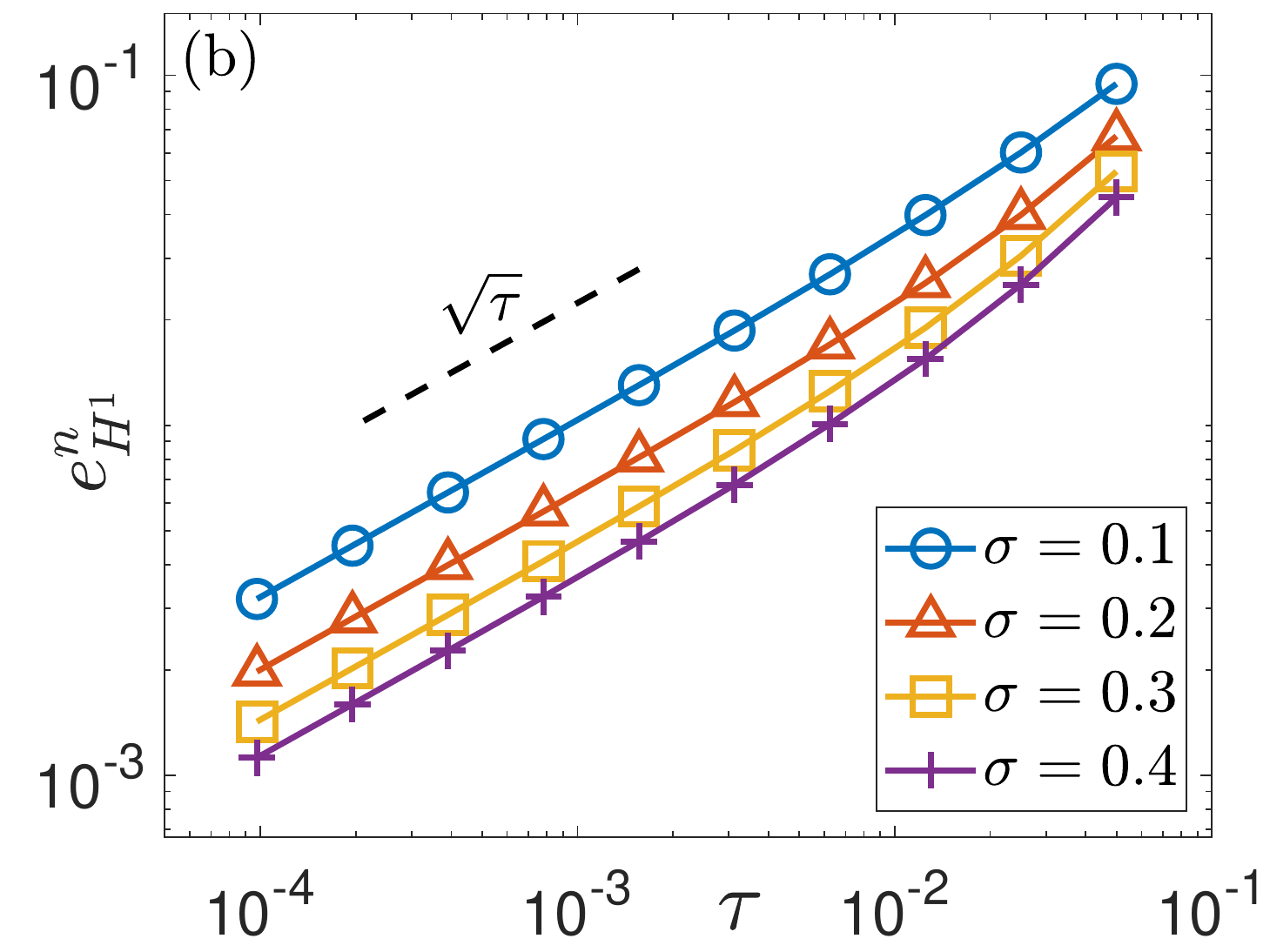}
	\caption{Temporal errors of the EWI for the NLSE \cref{eq:NLSE_semi-smooth} with Type I initial datum \cref{typeI_ini}: (a) $L^2$-norm errors, and (b) $H^1$-norm errors. }
	\label{fig:conv_dt_semi_smooth_H2_ini}
\end{figure}

Then we consider the Type II smooth initial datum \cref{typeII_ini}. \cref{fig:comp_FS_FP_semi_smooth_smooth_ini} shows the spatial error in $ L^2 $- and $ H^1 $-norm of the EWI-FS (solid lines) and the EWI-EFP (dotted lines) method for $ \sigma = 0.1 $ with the Type II initial datum. We can observe that the convergence orders in $ H^1 $-norm of the EWI-FS (solid lines) and the EWI-EFP (dashed lines) are almost the same (roughly 2.5), though the value of the error of the EWI-FS is smaller than the EWI-EFP. While the convergence order in $ L^2 $-norm of the EWI-FS method is roughly 3.5, which is almost one order higher than that of the EWI-EFP method. This observation suggests that when the solution has better regularity, the Fourier spectral method outperforms the Fourier pseudospectral method for discretizing the low regularity nonlinearity.

\cref{fig:conv_dt_semi_smooth_smooth_ini} displays the temporal error in $ L^2 $- and $ H^1 $-norm of the EWI for different $ 0<\sigma<1 $ with the Type II initial datum. \cref{fig:conv_dt_semi_smooth_smooth_ini} (a) and (b) show that the temporal convergence is first order in both $ L^2 $- and $ H^1 $-norm for all the four $ \sigma $. However, currently, we can only prove the first order $ H^1 $-convergence in time under Assumption \cref{B} which holds only when $ \sigma \geq 1/2 $. Besides, as shown in Figure 5.3 in \cite{bao2023}, for the time-splitting methods, we can observe first order convergence in $ H^1 $-norm only when $ \sigma \geq 1/2 $, which suggests that the EWI may be better than the time-splitting methods when the nonlinearity is of low regularity.

The results in \cref{fig:comp_FS_FP_semi_smooth_smooth_ini,fig:conv_dt_semi_smooth_smooth_ini} confirm our optimal $ H^1 $-norm error bound for the NLSE with low regularity nonlinearity, but also indicates that Assumption \cref{B} may be relaxed.

\begin{figure}[htbp]
	\centering {\includegraphics[width=0.475\textwidth]{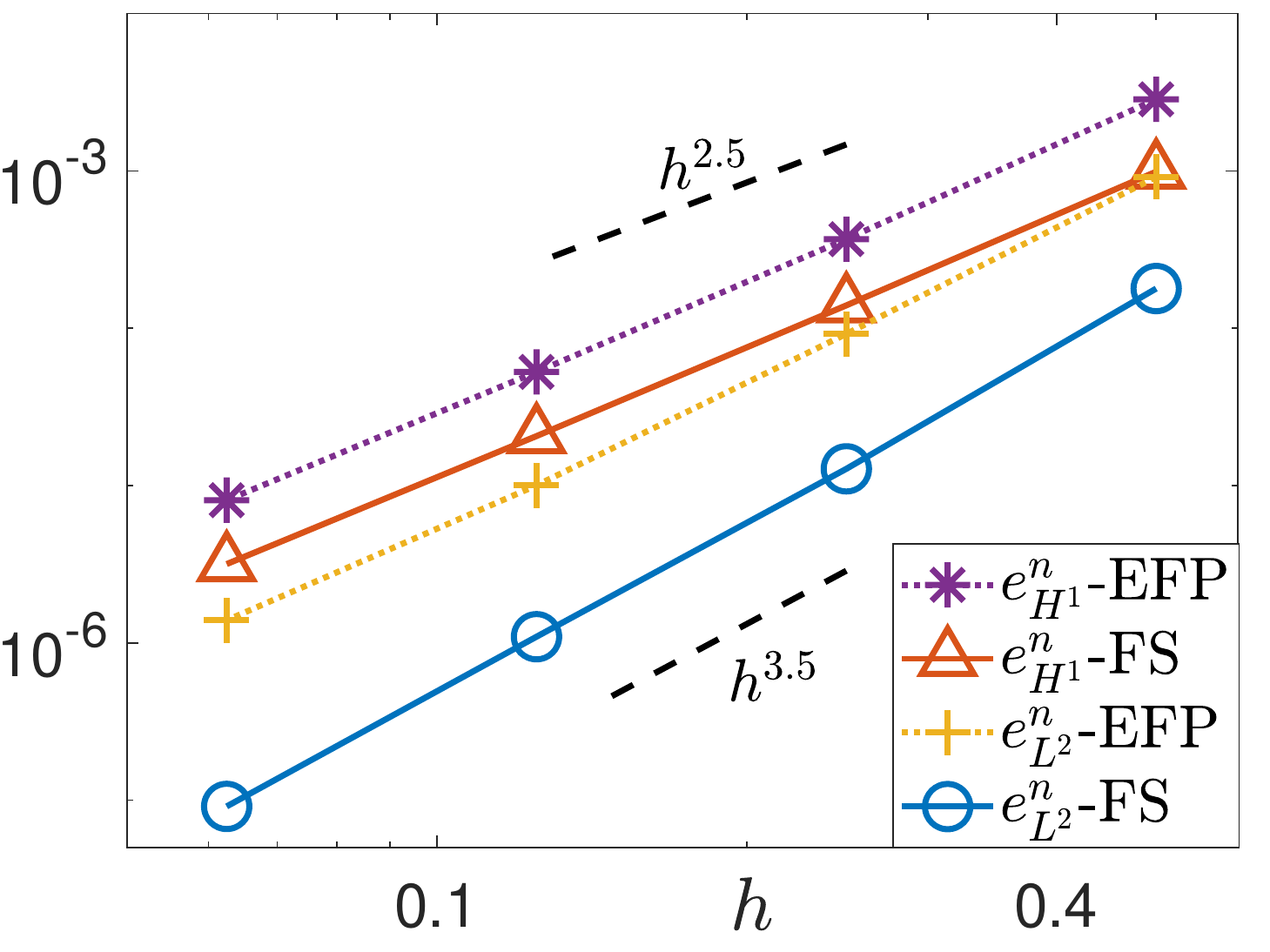}}
	\caption{Comparison of the Fourier spectral and pseudospectral discretizations of the nonlinear term in \cref{eq:NLSE_semi-smooth} with $ \sigma = 0.1 $ and Type II initial datum \cref{typeII_ini}.}
	\label{fig:comp_FS_FP_semi_smooth_smooth_ini}
\end{figure}

\begin{figure}[htbp]
	\centering
	{\includegraphics[width=0.475\textwidth]{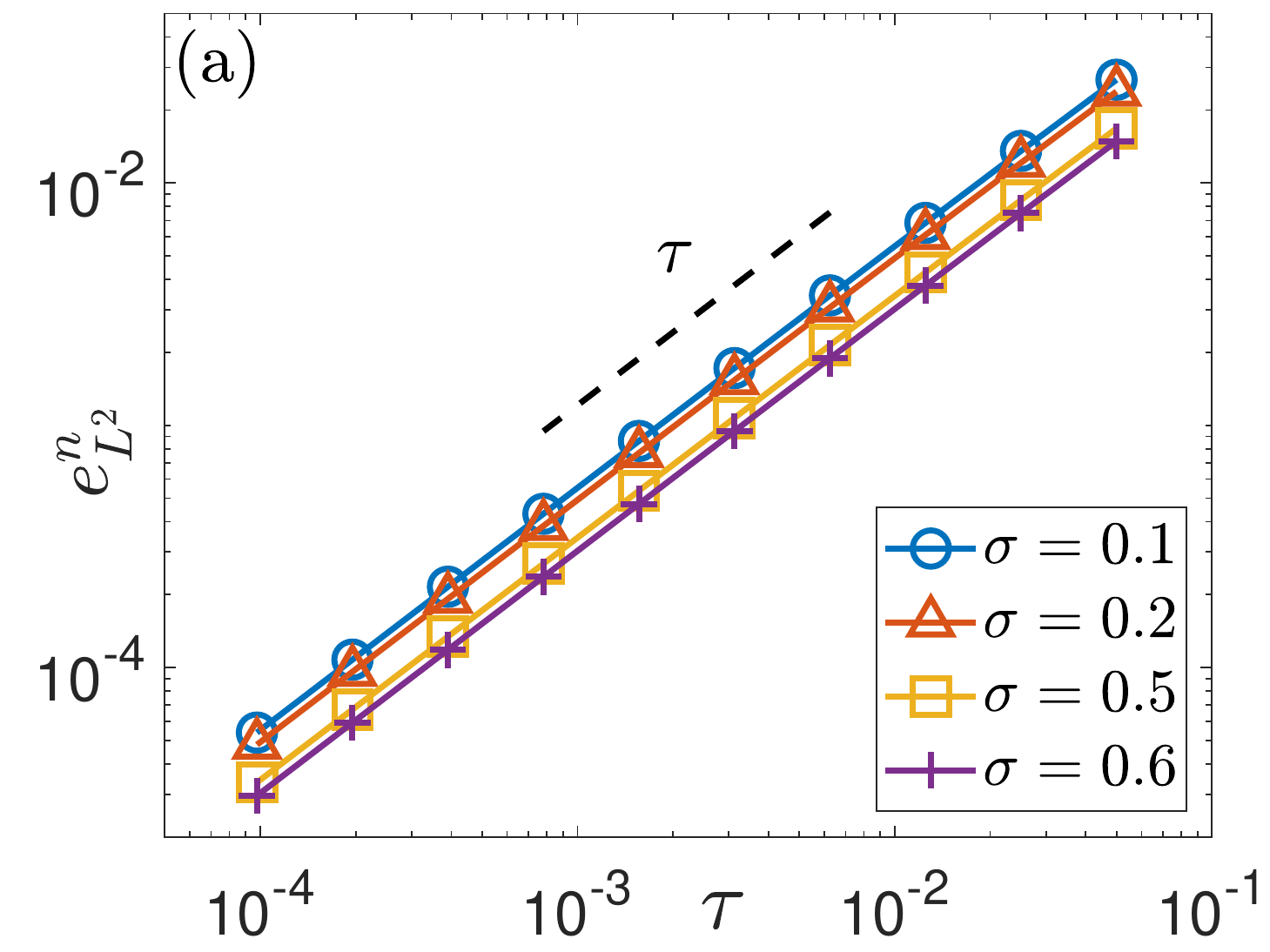}}
	{\includegraphics[width=0.475\textwidth]{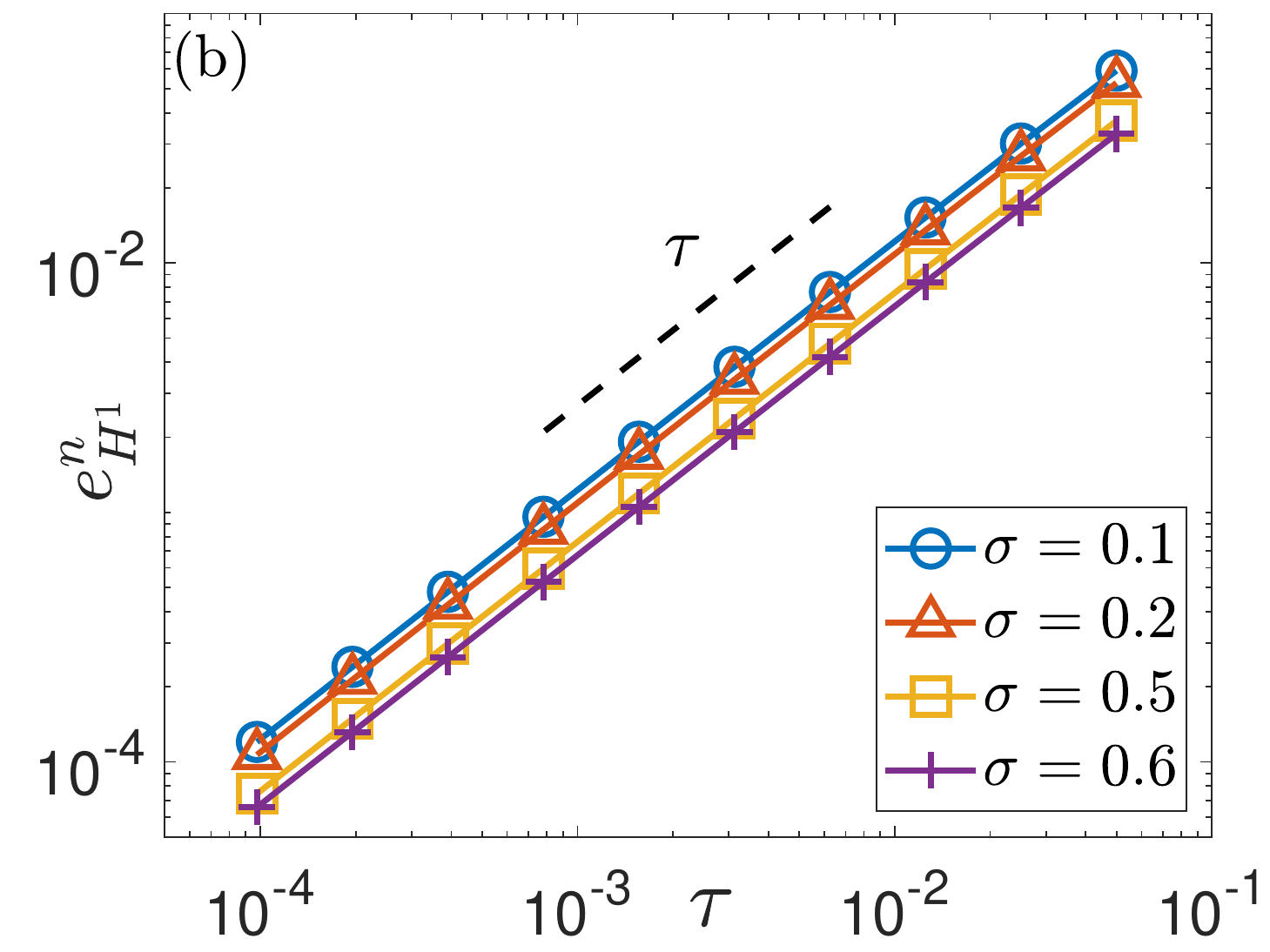}}
	\caption{Temporal errors of the EWI for the NLSE \cref{eq:NLSE_semi-smooth} with Type II initial datum \cref{typeII_ini}: (a) $L^2$-norm errors, and (b) $H^1$-norm errors.}
	\label{fig:conv_dt_semi_smooth_smooth_ini}
\end{figure}

\subsection{For the NLSE with low regularity potential}
In this subsection, we only consider the cubic NLSE with low regularity potential as
\begin{equation}\label{eq:NLSE_Linfty_poten}
	i \partial_t \psi(x, t) = -\Delta \psi(x, t) + V(x)\psi(x, t) - |\psi(x, t)|^2 \psi(x, t), \quad \vx \in \Omega, \quad t>0,
\end{equation}
where $ V $ is chosen as either $ V_1 \in L^\infty(\Omega) $ or $ V_2 \in W^{1, 4}(\Omega) $ defined as
\begin{equation}\label{eq:potential}
	V_1(x) = \left\{
	\begin{aligned}
		&-4, &x \in (-2, 2) \\
		&0, &\text{otherwise}
	\end{aligned}
	\right., \qquad V_2(x) = |x|^{0.76}, \qquad x \in \Omega.
\end{equation}

We shall test the convergence orders for the NLSE \cref{eq:NLSE_Linfty_poten} with $ V=V_1 $ and $ \psi_0 \in H^2(\Omega) $, and $ V = V_2 $ and $ \psi_0 \in H^3(\Omega) $, respectively. The 'exact' solutions are computed by the EWI-EFP method with $ \tau = \tau_\text{e} := 10^{-6} $ and $ h = h_\text{e} := 2^{-9} $. {When test the spatial errors, we fix the time step size $\tau = \tau_e$, and when test the temporal errors, we fix the mesh size $h = h_e$.}

We start with the spatial error and compare the performance of the extended Fourier pseudospectral method and the standard Fourier pseudospectral (FP) method which can be obtained by replacing $ \widehat{(V I_N \psihhn{n})_l} $ with $ \widetilde{(V \psihhn{n})_l} $ in \cref{EWI-EFP}. We remark here that, since the nonlinearity is smooth in \cref{eq:NLSE_Linfty_poten}, the results of the EWI-FS method are almost the same as those of the EWI-EFP method.

\cref{fig:Linfty_poten_H2_ini} (a) shows the spatial error in $ L^2 $- and $ H^1 $-norm of the EWI-EFP method (solid lines) and the EWI-FP method (dotted lines) with $ V = V_1 \in L^\infty(\Omega)$ given in \cref{eq:potential} and $ \psi_0 \in H^2(\Omega) $ given in \cref{typeI_ini}. We can observe that the EWI-EFP is second order convergent in $ L^2 $-norm and first order convergent in $ H^1 $-norm in space. However, the spatial convergence order of the EWI-FP method is only first order in both $ L^2 $- and $ H^1 $-norm, and the value of the error is much larger. This implies that when discretizing purely $ L^\infty $-potential, the extended Fourier pseudospectral method is much better than the standard Fourier pseudospectral method. \cref{fig:Linfty_poten_H2_ini} (b) plots the temporal convergence of the EWI in $ L^2 $- and $ H^1 $-norm with the Type I initial datum. We can observe that the EWI is first order convergent in $ L^2 $-norm and half order convergent in $ H^1 $-norm in time.

The results in \cref{fig:Linfty_poten_H2_ini} validate our optimal $ L^2 $-norm error bound for the NLSE with $L^\infty$-potential and demonstrate that it is sharp.

\begin{figure}[htbp]
	\centering
	{\includegraphics[width=0.475\textwidth]{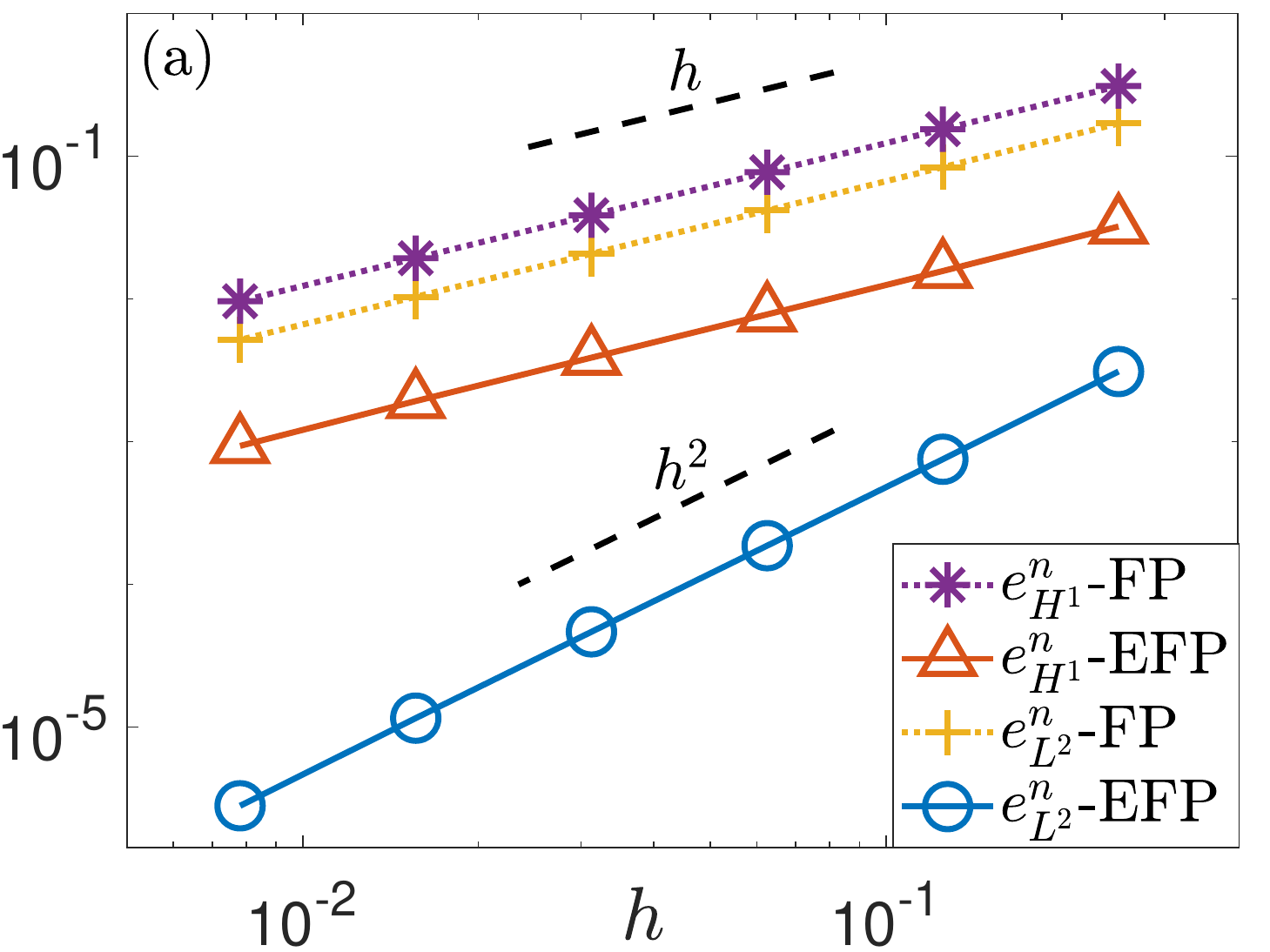}}
	{\includegraphics[width=0.475\textwidth]{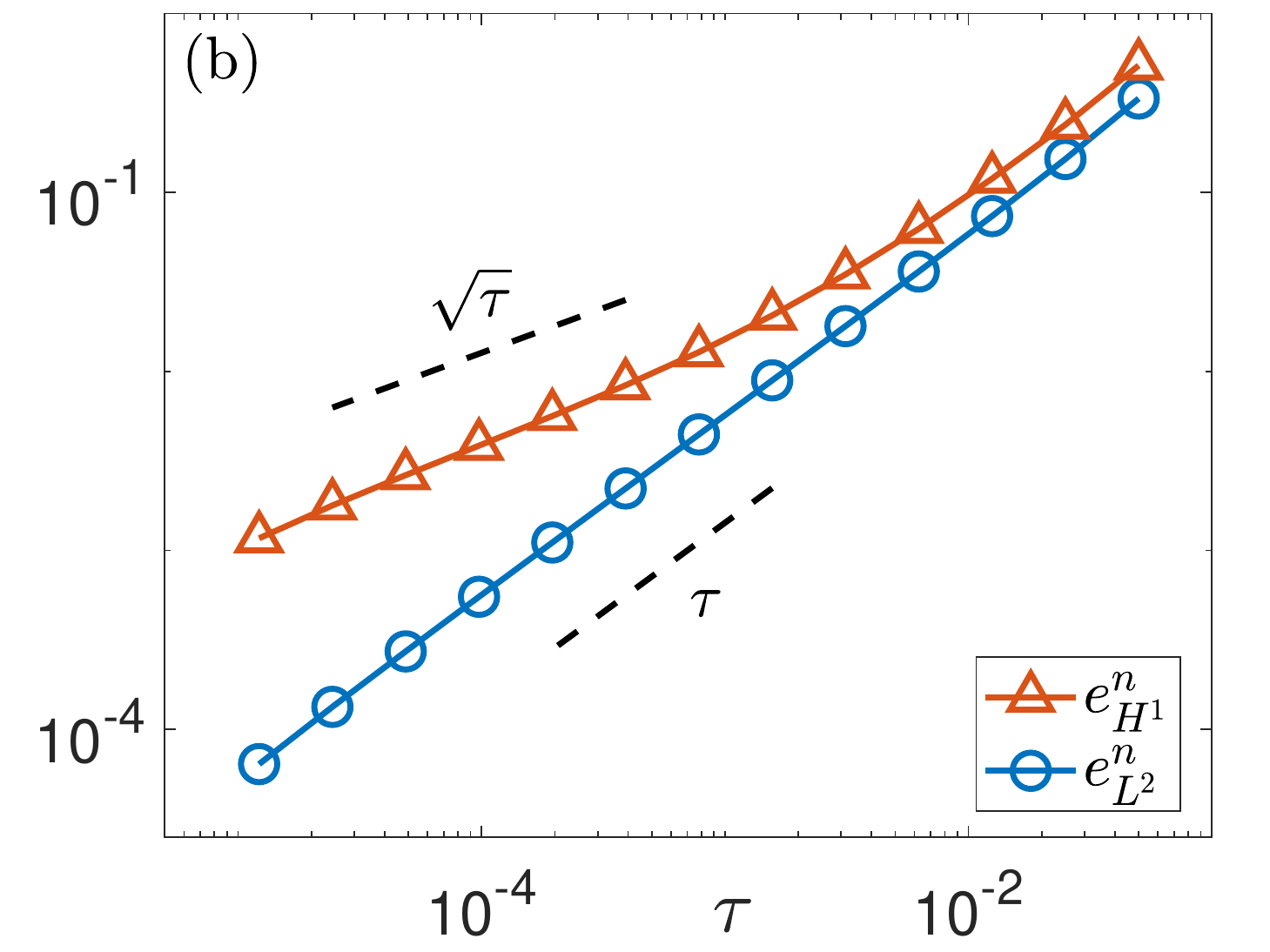}}
	\caption{Convergence tests of the EWI for \cref{eq:NLSE_Linfty_poten} with $ V=V_1 \in L^\infty(\Omega) $ and $ \psi_0 \in H^2(\Omega) $: (a) spatial errors
of the Fourier spectral and pseudospectral discretizations for the linear potential,
and  (b) temporal errors in $L^2$-norm and $H^1$-norm. }
	\label{fig:Linfty_poten_H2_ini}
\end{figure}

\begin{figure}[htbp]
	\centering
	{\includegraphics[width=0.475\textwidth]{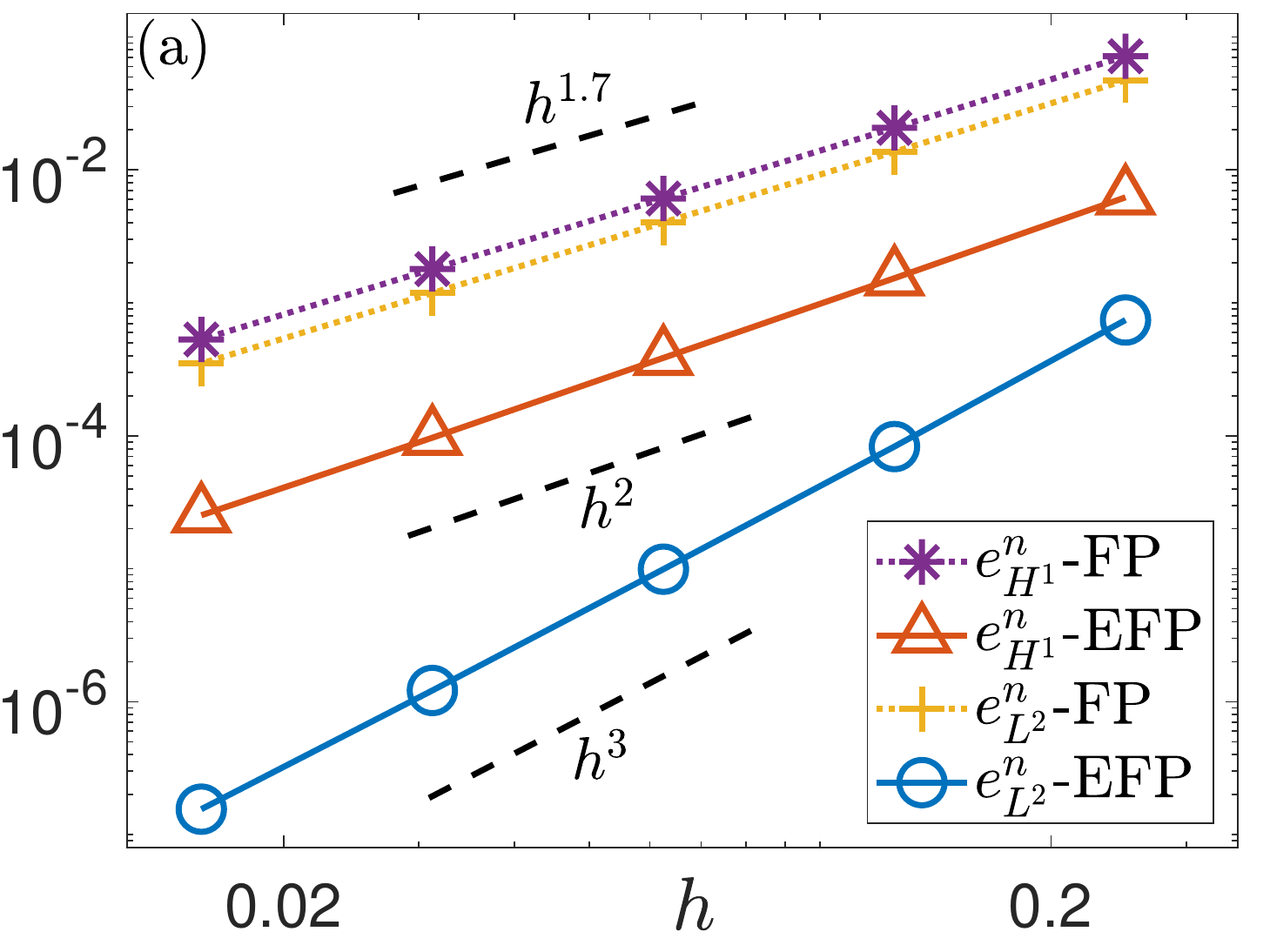}}
	{\includegraphics[width=0.475\textwidth]{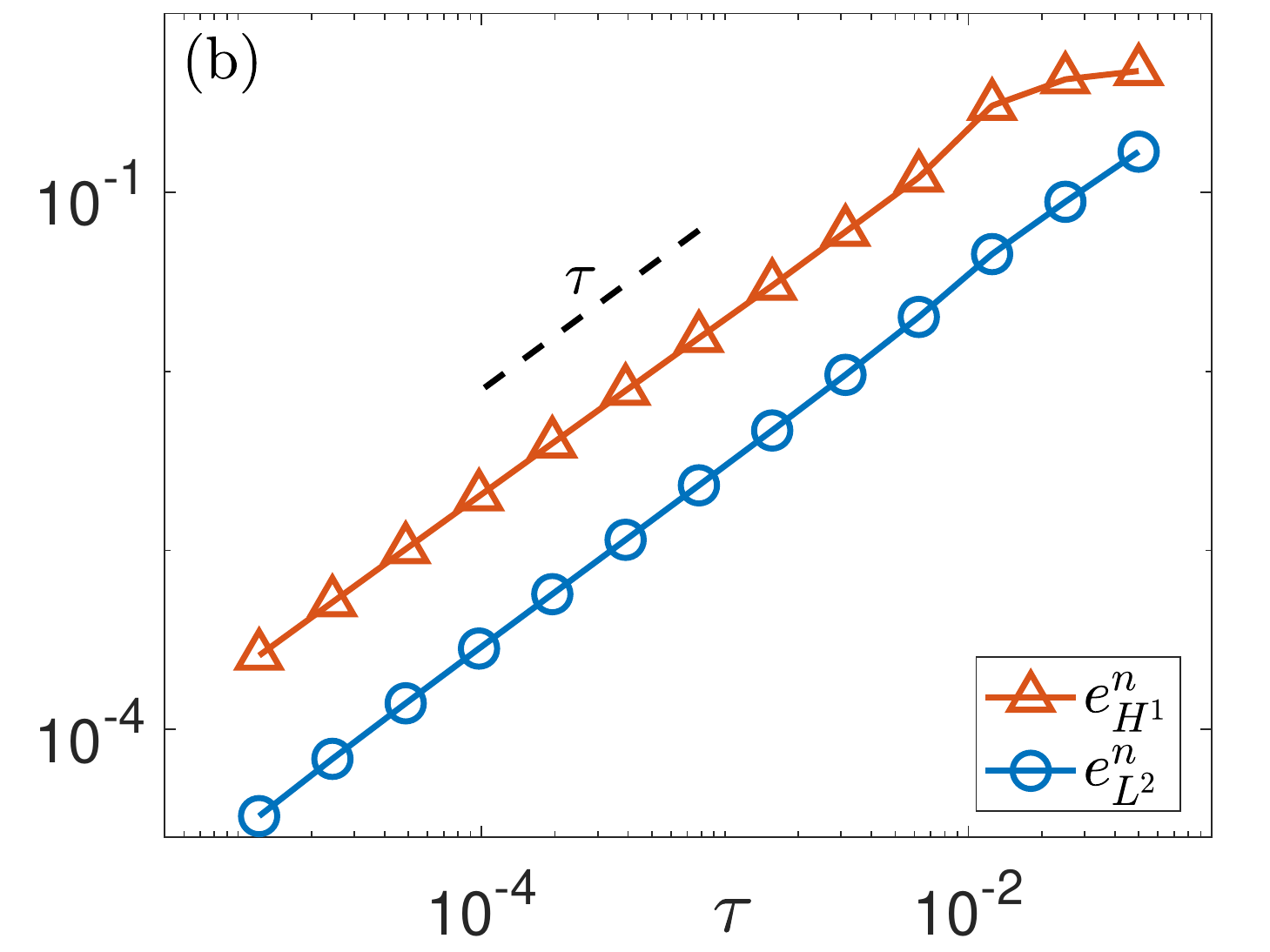}}
	\caption{Convergence tests of the EWI for \cref{eq:NLSE_Linfty_poten} with
$ V=V_2 \in W^{1, 4}(\Omega) $ and $ \psi_0 \in H^3(\Omega) $: (a) spatial errors
of the Fourier spectral and pseudospectral discretizations for the linear potential,
and  (b) temporal errors in $L^2$-norm and $H^1$-norm. }
	\label{fig:W14_poten_H3_ini}
\end{figure}

\cref{fig:W14_poten_H3_ini} (a) shows the spatial error in $ L^2 $- and $ H^1 $-norm of the EWI-EFP method (solid lines) and the EWI-FP method (dotted lines) with $ V = V_2 \in W^{1, 4}(\Omega)$ given in \cref{eq:potential} and $ \psi_0 \in H^3(\Omega) $ given by $ \psi_0(x) = (1+|x|^{2.51})e^{-x^2/2} $. We can observe that the EWI-EFP is third order convergent in $ L^2 $-norm and second order convergent in $ H^1 $-norm in space. However, the spatial convergence order of the EWI-FP method is only $ 1.7 $ order in both $ L^2 $- and $ H^1 $-norm, and the value of the error is much larger. This implies again that the extended Fourier pseudospectral method is much better than the standard Fourier pseudospectral method when the potential is of low regularity. \cref{fig:W14_poten_H3_ini} (b) plots the temporal convergence of the EWI in $ L^2 $- and $ H^1 $-norm with the $ H^3 $ initial datum. We can observe that the EWI is first order convergent in $ H^1 $-norm in time for $ V \in W^{1, 4}(\Omega) $.

The results in \cref{fig:W14_poten_H3_ini} validate our optimal $ H^1 $-norm error bound for the NLSE with $ W^{1, 4} $-potential and demonstrate that it is sharp.

{

\subsection{Comparison with the time-splitting method}
In this subsection, we present some numerical results to compare the performance of the EWI and the time-splitting method applied to the NLSE with low regularity potential and nonlinearity. To be precise, we compare the EWI with the first-order Lie-Trotter time-splitting method with standard Fourier pseudospectral method for spatial discretization (abbreviated as TSFP in the following). Here, we fix $h = h_e$ and compare the temporal errors, roughly speaking, this is equivalent to do comparison for semi-discretization in time by different time integrators. 

First, we consider the NLSE \cref{eq:NLSE_semi-smooth} with low regularity nonlinearity $\sigma = 0.1 $ and the smooth initial datum \cref{typeII_ini}. In \cref{fig:comp_nonl}, we can observe that both the EWI and the TSFP are first order convergent in $L^2$-norm, although the value of the error of the TSFP method is smaller than the EWI. However, when measured in $H^1$-norm, the EWI is still first order convergent (although this is not covered by our error estimates as already mentioned in the discussion of \cref{fig:conv_dt_semi_smooth_smooth_ini}), but the error of the TSFP method fluctuates a lot, and leads to order reduction. 

\begin{figure}[htbp]
	\centering
	{\includegraphics[width=0.475\textwidth]{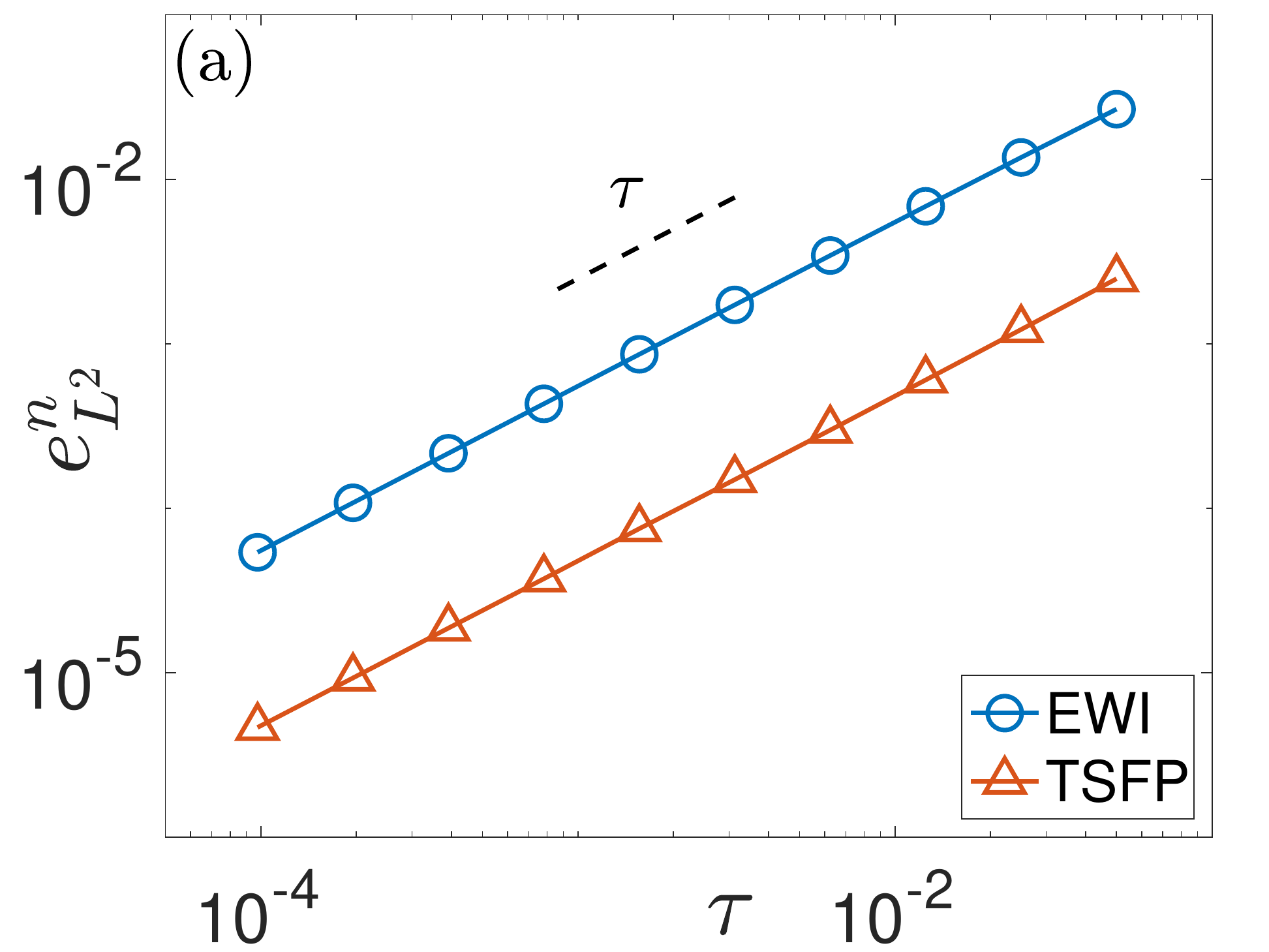}}
	{\includegraphics[width=0.475\textwidth]{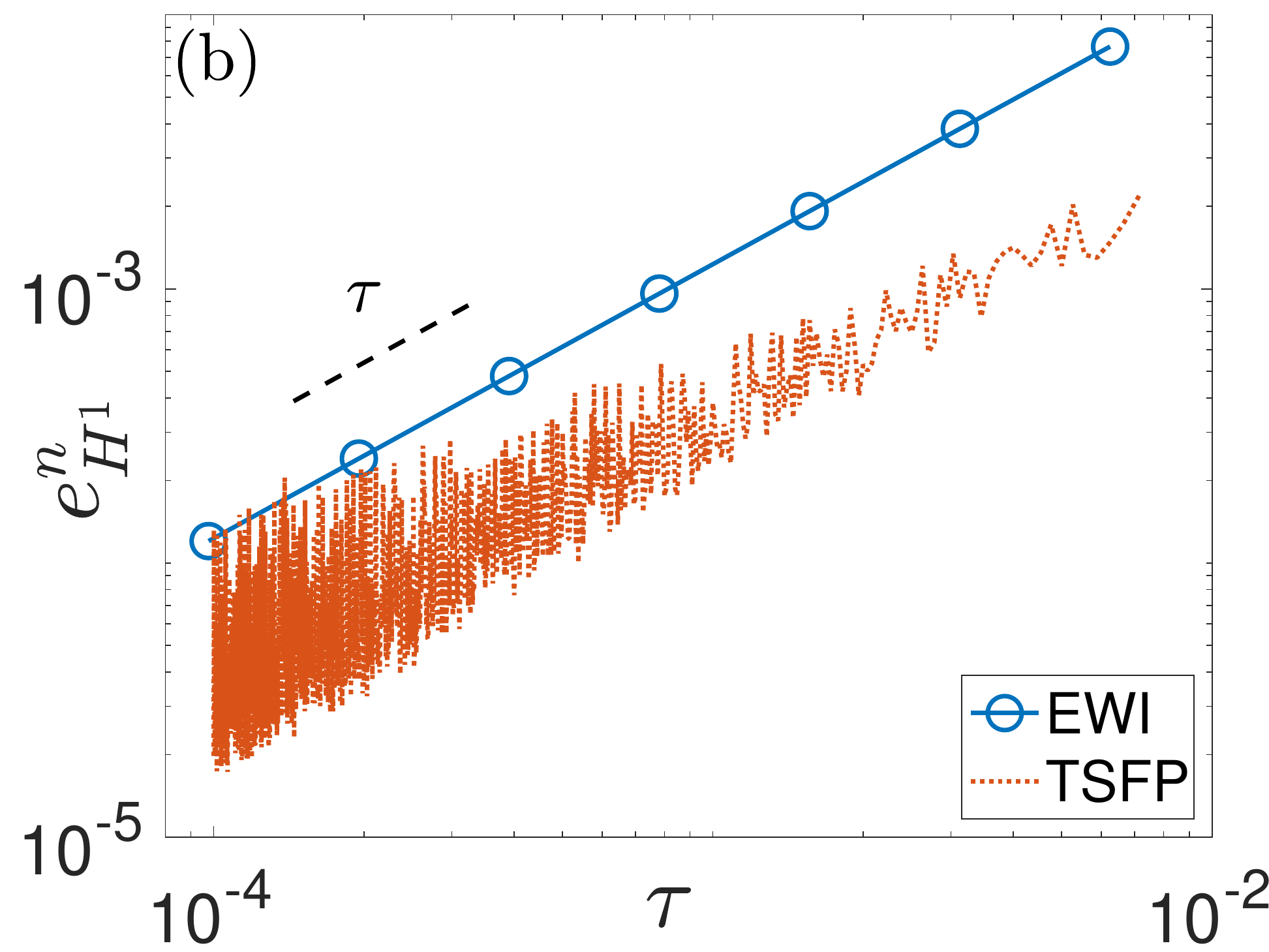}}
	\caption{comparison of EWI and LTFP for the NLSE \cref{eq:NLSE_semi-smooth} with $\sigma = 0.1$: (a) temporal errors in $L^2$-norm and (b) temporal errors in $H^1$-norm.}
	\label{fig:comp_nonl}
\end{figure}

Then we consider the NLSE \cref{eq:NLSE_Linfty_poten} with low regularity potential $V = V_1 \in L^\infty(\Omega)$ in \cref{eq:potential} and an $H^2$-initial data given in \cref{typeI_ini}. In \cref{fig:comp_Linfty_poten}, we can observe that the EWI is first order and half order convergent in $L^2$- and $H^1$-norm, respectively. However, both the $L^2$- and $H^1$-error of the TSFP method fluctuates drastically and suffer from sever order reduction. 

\begin{figure}[htbp]
	\centering
	{\includegraphics[width=0.475\textwidth]{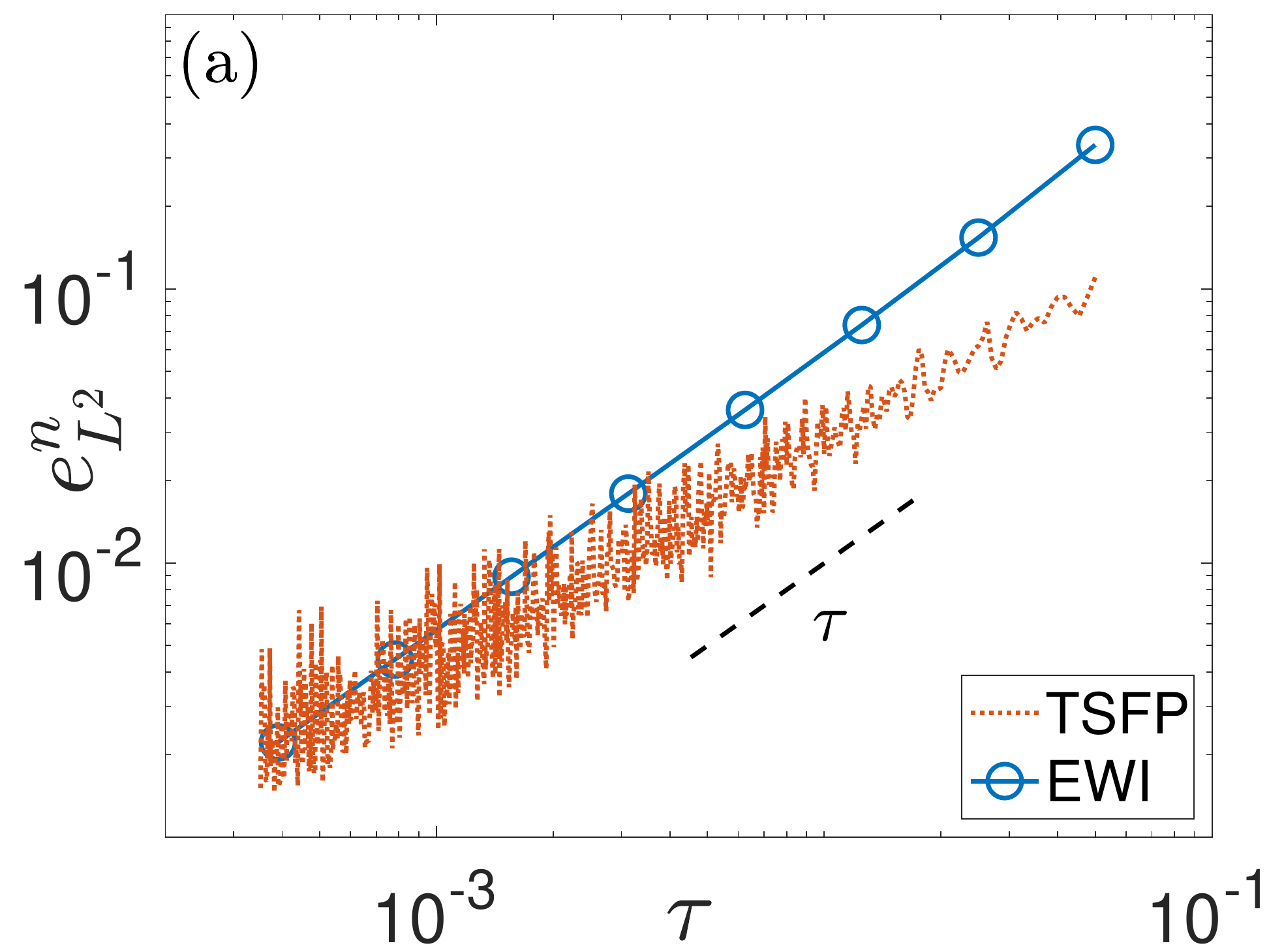}}
	{\includegraphics[width=0.475\textwidth]{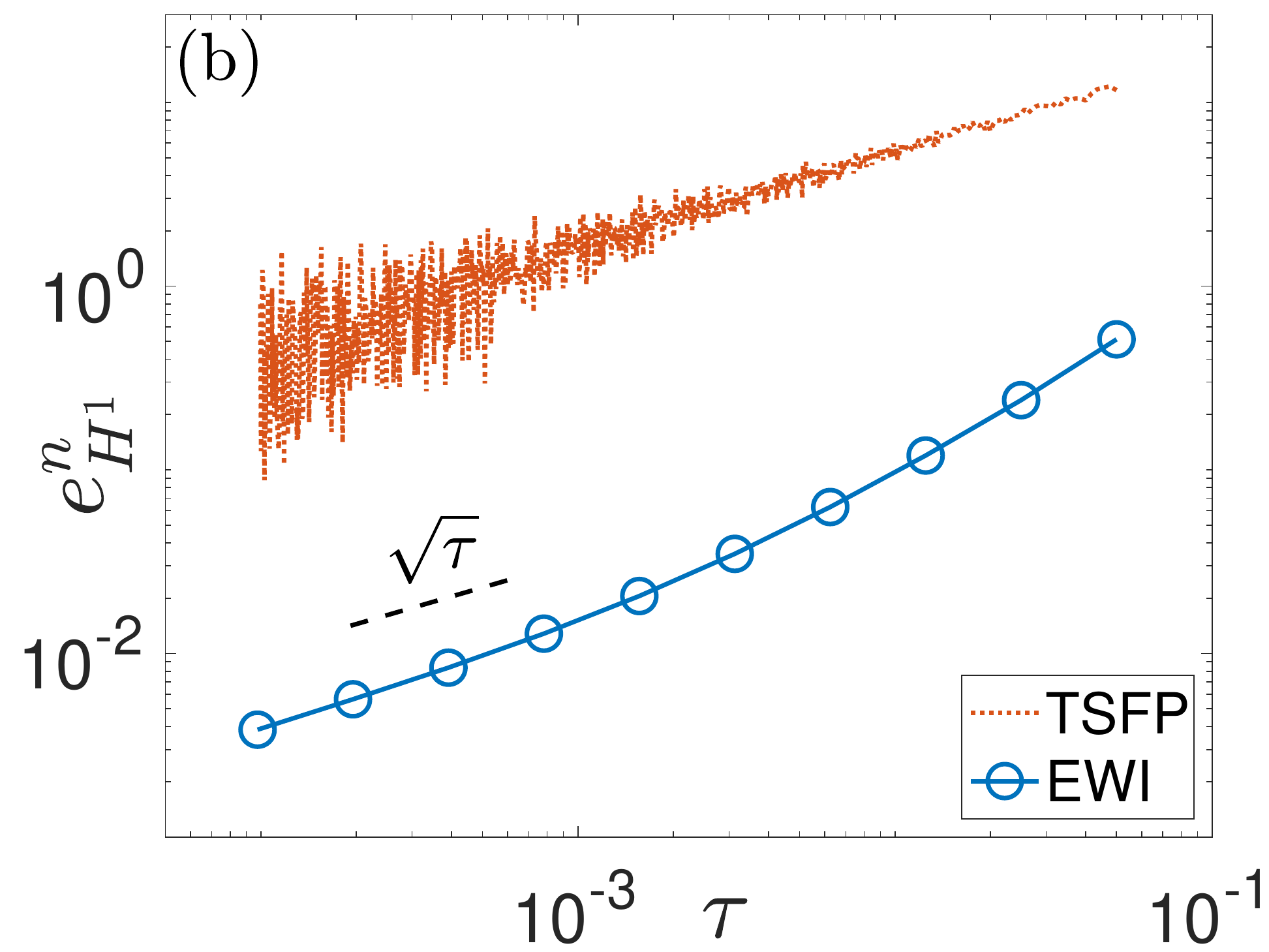}}
	\caption{comparison of EWI and LTFP for the NLSE \cref{eq:NLSE_Linfty_poten} with $V = V_1 \in L^\infty(\Omega)$ in \cref{eq:potential}: (a) temporal errors in $L^2$-norm and (b) temporal errors in $H^1$-norm.}
	\label{fig:comp_Linfty_poten}
\end{figure}

Based on the discussion above, we can conclude that in general, the EWI is better than the TSFP method when approximating the NLSE with low regularity potential and nonlinearity. However, the numerical results also necessitate the design and analysis of higher order and structure-preserving (e.g. time symmetric) EWIs for better error constant. This will be considered in our future work. 
}

\section{Conclusions}
We established optimal error bounds for the first-order Gautschi-type exponential wave integrator (EWI) applied to the nonlinear Schr\"odinger equation (NLSE) with $ L^\infty $-potential and/or locally Lipschitz nonlinearity under the assumption of $ H^2 $-solution. For the semi-discretization in time by the first-order Gautschi-type EWI, we proved an optimal $ L^2 $-norm error bound at $ O(\tau) $ and a uniform $ H^2 $-bound of the numerical solution. For the full discretization obtained from the semi-discretization by using the Fourier spectral method in space, we proved an optimal $ L^2 $-norm error bound at $ O(\tau + h^2) $ without any coupling condition between $ \tau $ and $ h $. For $ W^{1, 4} $-potential and a little more regular nonlinearity, under the assumption of $ H^3 $-solution of the NLSE, we proved optimal $ H^1 $-norm error bounds for both the semi-discrete and fully discrete schemes. As a by-product, we proposed an extended Fourier pseudospectral method to implement the full discretization when the potential is of low regularity and the nonlinearity is smooth,
in which the potential and nonlinearity were discretized by the Fourier
spectral method and the Fourier pseudospectral method, respectively.
The proposed numerical implementation has similar computational cost as the standard
Fourier pseudospectral method, but we can establish rigorous error bounds for this method. On the contrary, one cannot establish optimal error bounds for the
standard Fourier pseudospectral method for the NLSE when the potential is of low regularity, e.g. $V\in L^\infty$. In the future, we will consider even weaker potential, e.g. $V\in L^1$, including Coulomb potential and/or spatial/temporal Dirac delta potential.

\bibliographystyle{siamplain}

\end{document}